\documentclass[11pt, a4paper]{amsart}

\usepackage[utf8]{inputenc}
\usepackage[english]{babel}

\usepackage{enumerate}
\usepackage{color}
\usepackage{stmaryrd}
\usepackage[text={14.06cm,20.84cm},centering]{geometry}
\usepackage{marginnote} \usepackage{framed}

\usepackage{xcolor}
\usepackage{esint,amssymb,amsmath}
\usepackage{graphicx}
\usepackage{mathtools}
\usepackage[colorlinks=true, pdfstartview=FitV, linkcolor=blue, citecolor=blue, urlcolor=blue,pagebackref=false]{hyperref}
\usepackage{microtype}

\usepackage{bm}
\usepackage{scalerel} 
\usepackage{mathrsfs}
\usepackage{tikz}   
\usepackage[all]{xy} \xyoption{arc} \xyoption{color}
\usepackage{epsfig}
\usepackage{amsbsy}
\usepackage[font={footnotesize}]{caption}

\usepackage{enumitem}

\usepackage{comment}

\newcommand{\eps}{\varepsilon}
\newcommand{\wiener[1]}{\mathbb{A}^{#1}}

\newtheorem{theorem}{Theorem}
\newtheorem{proposition}{Proposition}
\newtheorem{lemma}[proposition]{Lemma}

\theoremstyle{remark}
\newtheorem{remark}[proposition]{Remark}

\theoremstyle{definition}
\newtheorem{definition}[proposition]{Definition}

\numberwithin{equation}{section}
\numberwithin{proposition}{section}
\title[Thin film approximations of one-phase unstable Muskat]{Rigorous thin film approximations of the one-phase unstable Muskat problem}
\author{Edoardo Bocchi and Francisco Gancedo}

\begin{document}

\begin{abstract}
This paper studies the one-phase Muskat problem driven by gravity and surface tension. The regime considered here is unstable with the fluid on top of a dry region. 
By a novel approach using a depth-averaged formulation, we derive two asymptotic approximations for this scenario. The lower order approximation is the classical thin film equation, while the higher order approximation provides a new refined thin film equation. We prove the optimal order of convergence in the shallowness parameter to the original Muskat solutions for both models with low-regular initial data.
\end{abstract}

\maketitle
\section{introduction}

We study in this paper the one-phase Muskat problem, which describes the motion of an incompressible viscous fluid in a porous medium moving by gravity and surface tension.
This common physical phenomenon is classically modeled through the Darcy's law \cite{Darcy1856} which states that forces acting on the fluid equal the flow velocity:
\begin{equation}\label{Darcy}
	\frac{\nu}{k}u=-\nabla p -g \rho\mathbf(0,1).
\end{equation}
Here $u$ is the fluid velocity,  $p$ is the fluid pressure, $\nu>0$ is the dynamic viscosity and $\rho>0$ is the density. The constants $k$ and $g$ are the permeability of medium and gravity, respectively. This law was first used by Muskat in \cite{Muskat} two deal with the dynamics of immiscible fluids in porous media. It arises in hydrogeology, in particular in the study of groundwater flows through aquifers.
This scenario has been widely studied, in particular for the mathematical analogies with the free boundary system given by a fluid in a vertical Hele-Shaw cell \cite{SaffmanTaylor1958}. 
It governs the motion of a viscous fluid confined between two vertical parallel flat plates in proximity to each other. The equivalence between the two problems holds by identifying the fixed distance between the plates in a Hele-Shaw cell to the the constant permeability of the homogeneous medium in the one-phase Muskat problem.\\
In this paper, we consider the unstable configuration in which an incompressible, homogeneous and viscous fluid lies over the dry region.
We denote by $u(t,x,z)$ the fluid velocity vector field, by $\zeta(t,x)$  the $2\pi$-periodic free surface elevation and by $p(t,x,z)$ the fluid pressure. Here $x$ is the horizontal variable and $z$ is the vertical (high or depth). We consider the two-dimensional case, suppressing the dependence in the other horizontal variable. The fluid domain is defined by $$\widetilde{\Omega}(t)=\{(x,z)\in \mathbb{T}\times \mathbb{R} \ |   -\zeta(t,x)<  z < H  \}$$ 
with upper and lower boundary given respectively by
$$\widetilde{\Gamma}_{\rm top}=\{(x,H) |  \ x\in\mathbb{T} \}, \qquad \widetilde{\Gamma}(t)=\{(x,-\zeta(t,x)) |  \ x\in\mathbb{T} \}.$$
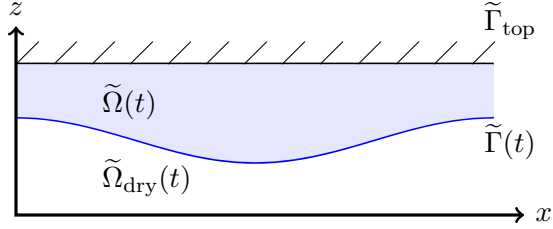
\begin{figure}[h]
	\centering
	\begin{tikzpicture}[domain=0:2*pi, scale=1]
		\draw[color=black] plot (\x,{0.3*cos(\x r)+1}); 
		\draw[very thick, smooth, variable=\x, blue] plot (\x,{0.3*cos(\x r)+1}); 
		\draw[very thick, smooth, variable=\x, black] plot (\x,2); 
		\fill[blue!10] plot[domain=0:2*pi] (\x,2) -- plot[domain=2*pi:0] (\x,{0.3*cos(\x r)+1});
		\draw[very thick,<->] (2*pi+0.4,0) node[right] {$x$} -- (0,0) -- (0,2.5) node[above] {$z$};
		\coordinate[label=above:{$\widetilde{\Gamma}(t)$}] (A) at (6.5,0.60);
		\coordinate[label=above:{$\widetilde{\Gamma}_{\text{\rm top}}$}] (A) at (6.5,2.2);
		\node[right] at (1,1.5) {$\widetilde{\Omega}(t)$};
		\node[right] at (1,0.5) {$\widetilde{\Omega}_{\text{dry}}(t)$};
		\draw[-] (0,2) -- (0.3, 2.3);
		\draw[-] (0.5,2) -- +(0.3, 0.3);
		\draw[-] (1,2) -- +(0.3, 0.3);
		\draw[-] (1.5,2) -- +(0.3, 0.3);
		\draw[-] (2,2) -- +(0.3, 0.3);
		\draw[-] (2.5,2) -- +(0.3, 0.3);
		\draw[-] (3,2) -- +(0.3, 0.3);
		\draw[-] (3.5,2) -- +(0.3, 0.3);
		\draw[-] (4,2) -- +(0.3, 0.3);
		\draw[-] (4.5,2) -- +(0.3, 0.3);
		\draw[-] (5,2) -- +(0.3, 0.3);
		\draw[-] (5.5,2) -- +(0.3, 0.3);
		\draw[-] (6,2) -- +(0.3, 0.3);
	\end{tikzpicture} 
	\caption{Physical setting}
\end{figure}

The one-phase unstable Muskat problem is described by the following system:
\begin{equation}\label{HSunstable0}
	\begin{cases}\begin{aligned}
			&\frac{\nu}{k}u=-\nabla p -g \rho\mathbf{e}_z \qquad  &\mbox{in}\quad \widetilde{\Omega}(t)\\
			&\nabla \cdot u=0\qquad &\mbox{in}\quad \widetilde{\Omega}(t)\\
			&\partial_t (-\zeta) = u\cdot N_{-\zeta} \qquad &\mbox{on}\quad \widetilde{\Gamma}(t)\\
			&p=\gamma \kappa_{-\zeta} \qquad &\mbox{on}\quad \widetilde{\Gamma}(t)\\
			&u\cdot \mathbf{e}_z=0  \qquad &\mbox{on} \quad \widetilde{\Gamma}_{\rm top} 
		\end{aligned}
	\end{cases}
\end{equation} where $\mathbf{e}_z=(0,1)$ is the unit vector parallel to the vertical axis, $N_{-\zeta}=(\partial_x(-\zeta),-1)^T$ is the outgoing normal vector to $\widetilde{\Gamma}(t)$, $\gamma$ is the surface tension coefficient and  $\kappa_{-\zeta}$ is the mean surface curvature defined by
$$\kappa_{-\zeta}= \frac{\partial_{xx}(-\zeta)}{(1+ (\partial_x(-\zeta))^2)^{3/2}}.$$ 
The boundary condition for the pressure at the surface in \eqref{HSunstable0} is given by the Young-Laplace law for the pressure drop at the interface between two immiscible fluids considering the  exterior atmospheric pressure  equals to zero for simplicity.

The configuration of the one-phase unstable Muskat problem naturally appears in hydrology, in particular when dealing with groundwater. Indeed, below the soil there are additional layers that contain water plus air in the pores. At sufficient depth, one can find a capillarity zone, where water is drawn upward by surface forces and this is in the form of a film around the soil grains. Hence, capillarity effects play a key role in the stabilization of the Rayleigh-Taylor instability (see below) due to the unstable configuration of a thin film fluid layer in the capillarity zone of the soil. This motivates our goals in this paper: derivation of lubrication approximations in the unstable case and the study of the influence of the surface tension on the thin film asymptotic models.

\subsection{Previous results} In the mathematical literature, this scenario has attracted a lot of attention, in particular because it is a classical free boundary problem with wide applicability \cite{Bear1988}. In the case of surface tension and without gravity interaction ($g=0$), local-in-time existence for large data was first proved \cite{DuchonRobert1984,Chen1993,EscherSimonett1997}. See \cite{HQNguyen2019} where the regularity of the initial data is at any subcritical level in terms of Sobolev spaces. Capillarity force moving the fluid provides surface tension as a high-order parabolic character to the system, giving local-in-time existence. For small initial data, the parabolicity is stronger than the nonlinearity, giving global existence \cite{Chen1993,ConstantinPugh1993,YeTanveer2011} for solutions close to the circle or near planar. Considering gravity, when the fluid is on top, large initial data produce local in-time instabilities \cite{GHS2007}. This fact is also given when a less viscosity fluid is pushing a more viscous one, producing fingering for small surface tension \cite{Otto1997,EscherMatioc2011}. These are Rayleigh-Taylor instabilities, also known as Saffman-Taylor instabilities in the Muskat framework. Recently, these gravity unstable scenarios have been proved to have global existence for near planar solutions, giving instant time smoothing \cite{GG-BS2020}. Bubble shape solutions are particular unstable due to gravity movement and therefore due to viscosity. However, they have special structure giving globally in-time regularity for small initial slopes compared to the surface tension coefficient \cite{GG-JPS2019Bubble}.

Without surface tension, there are recent results showing very interesting behavior of the solutions from the fluid dynamics point of view. Non-graph stable regimes provide blow-up scenarios where two particles on the free boundary collide in finite time \cite{CCFG2016,CordobaPernas-Castano2017}. On the other hand, for the graph stable scenario, there are maximum principles for the high and the slope \cite{Kim2003,AlazardOneFluid2019}. These facts have been used to give well-posedness for an arbitrary slope \cite{DGN2021}. The solutions have been shown to do not smooth instantly \cite{APW2023}. For small slope, there is instant smoothing in the stable regime and ill-posedness in the unstable regime \cite{GG-JPS2019}.
In the stable regime, as the surface tension coefficient goes to zero, capillarity solutions approach to solutions without surface tension \cite{Ambrose2014} with low regularity at the optimal decay rate \cite{FlynnNguyec2021}.

Some of these features only hold in the one-phase case. In the two-phase scenario (two-immiscid fluid interaction), there is no particle collision as in the the one-phase case \cite{GancedoStrain2014}. Initial graphs can turn to nongraph in finite time \cite{CCFGL-F2012}. See \cite{Sema2017} for a review and \cite{GancedoLazar2022,AlazardHung2023,G-JG-SNP2022} for recent developments of the two-phase case.      

Viscous thin fluid films and lubrication approximations have been widely studied in fluid dynamics in the last three decades. The most studied asymptotic model is the classical thin film equation, a fourth-order degenerate parabolic equation. Both the existence of global weak solutions \cite{BF1990, BBD-P1995, BP1996} and singularity formation \cite{Con1993, Con2018} of the thin film equation were investigated three decades ago. Since these pioneering works, many authors focused their attention on different configurations of the thin film equation in order to get a complete and deeper understanding of the dynamics described by this approximated equation.
The equations for the motion of two thin fluid films in a porous medium were derived from the three-phase Muskat problem \cite{EMM2012}. The stable case and the influence of both capillary and gravitational effects were considered.
 In the case of purely gravity-driven motion, the local existence of classical solutions in the Sobolev space $H^2$ and exponential stability of steady state solutions were determined in \cite{EMM2012}, while the existence of nonnegative global weak solutions and exponential convergence towards equilibria in $L^2$ were obtained in \cite{ELM2011}. In the case of purely capillarity driven motion,  existence of nonnegative global weak solutions in $H^1$ is given using a priori estimates \cite{Mat2012}.
 In the case of both gravity and capillarity driven motion, local existence and asymptotic stability of steady states were proved in $H^4$ in the stable case \cite{EM2013}. Existence of global $L^2$ weak solutions was obtained on the full line using a gradient flow theory \cite{LM2013, LM2014} and it was recently proved on the torus in low regularity Wiener spaces together with propagation of Sobolev regularity \cite{BG-B2019}.
While the literature of the existence theory is wide and well-known, the rigorous justification of the lubrication approximations has not been deeply investigated. Indeed, to the best of our knowledge, the unique study dealing with this aspect is \cite{MP2012}, where the authors derived a cascade of thin film equations of different orders and rigorously justified them in highly regular Sobolev spaces. They considered a purely capillarity driven motion, the approximated equations derived are not explicit and moreover are obtained using two ansatzs, on the velocity potential and on the free surface.

\subsection{Main results}
The aim of this work is twofold. First, we derive lubrication approximations for the one-phase Muskat problem in a thin film regime in the unstable case, where the fluid lies over a dry region, taking into account the gravitational effects. To do that, exploiting the incompressibility condition, we employ a depth-averaged formulation used in the derivation of asymptotic models of water waves \cite{Lan17, Boc20} to write the evolution equation for the free surface. This formulation permits to derive different explicit thin film asymptotic models with respect to the shallowness parameter $\mu$ in a direct manner after having introduced only one ansatz for the velocity potential. In the regime $\tfrac{1}{\rm Bo}=O(1)$ and at approximation $O(\mu)$, denoting by $\zeta$ the free surface, we recover the gravity-driven thin film equation 
\begin{equation}\label{tf}
\partial_t \zeta + \partial_x\left((1+\eps \zeta)(\partial_x \zeta + \tfrac{1}{\mathrm{Bo}}\partial_{xxx} \zeta)\right)=0\\[5pt]
\end{equation} 
where $\eps$ is the nonlinearity parameter and $\rm Bo$ is the Bond number describing the ratio between gravitational and surface forces (see Section \ref{dimless_muskat} for the detailed definitions). Furthermore, in the different regime $\tfrac{1}{\rm Bo}=\tfrac{\sqrt{\mu}}{\rm bo}=O(\sqrt{\mu})$ and at approximation $O(\mu^{3/2})$ we obtain a refined gravity-driven thin film equation 
\begin{equation}\label{tf_ref}
\begin{aligned}
\partial_t  \zeta +\partial_x\left((1+\eps \zeta)(\partial_x \zeta + \tfrac{\sqrt{\mu}}{\mathrm{bo}}\partial_{xxx} \zeta)\right)+ \tfrac{\mu}{3}\partial_{xx}\left((1+\eps\zeta)^3 \partial_{xx}\zeta\right) =0,\\[10pt]
\end{aligned}
\end{equation} which, to the best of our knowledge, appears for the first time in the thin film literature.
Secondly, our goal is to rigorously justify the derived thin film asymptotic models, showing that their solutions are good approximations of the solution to the full one-phase unstable Muskat problem, whose existence has been recently obtained \cite{GG-BS2020}. To do that, we derive elliptic estimates for the velocity potential adapting the ones used in \cite{GG-BS2020}. We use them to get, respectively, in Theorem \ref{conv_theo} and Theorem \ref{conv_ref_theo}, error estimates for both the gravity-driven thin film and the refined gravity-driven thin film equations. Namely, we obtain for $\mathrm{Bo}<1$ and $\mu\in(0,1)$
\begin{equation*}
\|\zeta^{\rm MU}_\mu -\zeta^{\rm app}\|_{L^\infty([0,T]; \ \mathbb{A}^0_{\nu t})}+ 	 \tfrac{1}{64}\left(\tfrac{1}{\rm Bo}-1\right)\|\zeta^{\rm MU}_\mu -\zeta^{\rm app}\|_{L^1([0,T]; \ \mathbb{A}^4_{\nu t})}\leq  C\mu \\[10pt]
\end{equation*}
and for $\mathrm{bo}$ small enough and $\mu\in[\mu_0,1)$
\begin{equation*}
\|\zeta^{\rm MU}_\mu -\zeta^{\rm app}_{\rm ref}\|_{L^\infty([0,T]; \ \mathbb{A}^0_{\nu t})}+  \  \tfrac{1}{64}\left(\tfrac{\sqrt{\mu}}{\rm bo}+\tfrac{\mu}{3}-1\right)	\|\zeta^{\rm MU}_\mu -\zeta^{\rm app}_{\rm ref}\|_{L^1([0,T]; \ \mathbb{A}^4_{\nu t})}\leq  C\mu^{3/2}, \\[5pt]
\end{equation*}
where $\zeta^{\rm MU}_\mu$ is solution to the full one-phase unstable Muskat problem obtained in \cite{GG-BS2020} and $\zeta^{\rm app}$, $\zeta^{\rm app}_{\rm ref}$ are solutions to \eqref{tf} and \eqref{tf_ref} respectively. 
One of the novelties with respect to \cite{MP2012} is that we obtain asymptotic stability in low regularity Wiener space $\mathbb{A}^0_{\nu t}$ (see Section \ref{fun_sec} for the definition). Moreover, no assumptions are made on the approximated surface profiles $\zeta^{\rm app}$ and $\zeta^{\rm app}_{\rm ref}.$ The unstable configuration is responsible for the lower bound $\mu_0$ of the parameter $\mu$ in the convergence result. Contrarily, in the stable case the sign of the gravity term permits to remove that condition and to take arbitrarily small values of $\mu$. More precisely, we derive the refined thin film equation in the stable case
\begin{equation}\label{tf_stable_ref}
\begin{aligned}
\partial_t  \zeta +\partial_x\left((1+\eps \zeta)(-\partial_x \zeta + \tfrac{\sqrt{\mu}}{\mathrm{bo}}\partial_{xxx} \zeta)\right)- \tfrac{\mu}{3}\partial_{xx}\left((1+\eps\zeta)^3 \partial_{xx}\zeta\right) =0
\end{aligned}
\end{equation} and, denoting now by $\zeta^{\rm app}_{\rm ref}$ a solution to \eqref{tf_stable_ref}, we obtain for $\mathrm{bo}>0$ and $\mu\in(0,\mu_1)$ the error estimate
\begin{equation*}
\|\zeta^{\rm MU}_\mu -\zeta^{\rm app}_{\rm ref}\|_{L^\infty([0,T]; \ \mathbb{A}^0_{\nu t})}+  \  \tfrac{1}{64}\left(\tfrac{\sqrt{\mu}}{\rm bo}-\tfrac{\mu}{3}\right)	\|\zeta^{\rm MU}_\mu -\zeta^{\rm app}_{\rm ref}\|_{L^1([0,T]; \ \mathbb{A}^4_{\nu t})}\leq  C\mu^{3/2}. \\[5pt]
\end{equation*}
Finally, we point out that we investigate the most difficult unstable case and the most complete gravity-capillarity driven configuration. However, the analysis carried out in this paper, used for the derivation of asymptotic models and their rigorous justification, holds also in the purely capillarity-driven motion ($g=0$). 
Indeed, in the purely capillarity-driven scenario, one only needs to change the definition of the Bond number $\rm{Bo}$ and repeat the computations to obtain the same results (see Remark \ref{Rcapillary-drivenBond}). This configuration in the absence of gravity is the same of the well-known Hele-Shaw interface problem \cite{SaffmanTaylor1958}. Since the new term in the refined model is due to the presence of gravity, it turns out that under both orders of approximation $O(\mu)$ and $O(\mu^{3/2})$ we obtain the standard thin film equation
\begin{equation*}\label{tf_capillary}
\partial_t \zeta + \tfrac{1}{\mathrm{Bo}}\partial_x\left((1+\eps \zeta) \partial_{xxx} \zeta\right)=0
\end{equation*} whether $\tfrac{1}{\rm Bo}=O(1)$ or $O(\sqrt{\mu})$.

\subsection{Outline of the paper}
In Section \ref{muskat_sec} we introduce a depth-averaged formulation for the one-phase Muskat problem driven by gravity and surface tension in the unstable regime. In Section \ref{fun_sec} the functional framework in which the analysis is carried out is presented, introducing Wiener spaces, Wiener-Sobolev spaces, and structural propositions.
In Section \ref{lubri_sec} we deal with the lubrication approximations of the general problem and the thin film equations are derived after introducing different ansatzs on the fluid velocity potential. In Section \ref{ellipest_sec} we derive crucial elliptic estimates for the remainder velocity potential in both approximations. These estimates allow to obtain the optimal order of convergence for both asymptotics models in Section \ref{conv_sec}, detailed in Theorem \ref{conv_theo}, Theorem \ref{conv_ref_theo} and Theorem \ref{conv_ref_theo_stable}.

\section{One-phase unstable Muskat problem}\label{muskat_sec}
 After the change of variable $z'=-z$ and omitting the prime, the system \eqref{HSunstable0} can be recast as 
\begin{equation}\label{HSunstable}
\begin{cases}\begin{aligned}
&\frac{\nu}{k}u=-\nabla p + g \rho\mathbf{e}_z \qquad  &\mbox{in}\quad \Omega(t)\\
&\nabla \cdot u=0\qquad &\mbox{in}\quad \Omega(t)\\
&\partial_t \zeta = u\cdot N_\zeta \qquad &\mbox{on}\quad \Gamma(t)\\
&p=-\gamma \kappa_{\zeta} \qquad &\mbox{on}\quad \Gamma(t)\\
&u\cdot \mathbf{e}_z=0  \qquad &\mbox{on} \quad \Gamma_{\rm bot}
\end{aligned}
\end{cases}
\end{equation} 
where the fluid domain reads $$\Omega(t)=\{(x,z)\in \mathbb{T}\times \mathbb{R} \ |   -H<  z < \zeta(t,x)  \}$$ 
with upper and lower boundary given respectively by
$$\Gamma(t)=\{(x,\zeta(t,x)) |  \ x\in\mathbb{T} \}, \qquad \Gamma_{\rm bot}=\{(x,-H) |  \ x\in\mathbb{T} \}, $$
and $N_\zeta=(-\partial_x\zeta,1)^T$ is the outgoing normal vector to $\Gamma(t)$.

\begin{figure}[h]
	\centering
	\begin{tikzpicture}[domain=0:2*pi, scale=1] 
	\draw[color=black] plot (\x,{0.3*sin(\x r)+1}); 
	\draw[very thick, smooth, variable=\x, blue] plot (\x,{0.3*sin(\x r)+1}); 
	\fill[orange!10] plot[domain=0:2*pi] (\x,0) -- plot[domain=2*pi:0] (\x,{0.3*sin(\x r)+1});
	\draw[very thick,<->] (2*pi+0.4,0) node[right] {$x$} -- (0,0) -- (0,2) node[above] {$z$};
	\coordinate[label=above:{$\Gamma_{\text{b}}(t)$}] (A) at (6.8,0);
	\coordinate[label=above:{$\Gamma(t)$}] (A) at (6.8,0.72);
	\node[right] at (1,1.8) {$\Omega_{\text{dry}}(t)$};
	\node[right] at (1,0.5) {$\Omega(t)$};
	\node[right] at (2.5,0.3) {$u(x,z,t),p(x,z,t)$};
	\draw[-] (0,-0.3) -- (0.3, 0);
	\draw[-] (0.5,-0.3) -- +(0.3, 0.3);
	\draw[-] (1,-0.3) -- +(0.3, 0.3);
	\draw[-] (1.5,-0.3) -- +(0.3, 0.3);
	\draw[-] (2,-0.3) -- +(0.3, 0.3);
	\draw[-] (2.5,-0.3) -- +(0.3, 0.3);
	\draw[-] (3,-0.3) -- +(0.3, 0.3);
	\draw[-] (3.5,-0.3) -- +(0.3, 0.3);
	\draw[-] (4,-0.3) -- +(0.3, 0.3);
	\draw[-] (4.5,-0.3) -- +(0.3, 0.3);
	\draw[-] (5,-0.3) -- +(0.3, 0.3);
	\draw[-] (5.5,-0.3) -- +(0.3, 0.3);
	\draw[-] (6,-0.3) -- +(0.3, 0.3);
	\end{tikzpicture} 
	\caption{Physical setting after the change of coordinates}
\end{figure}
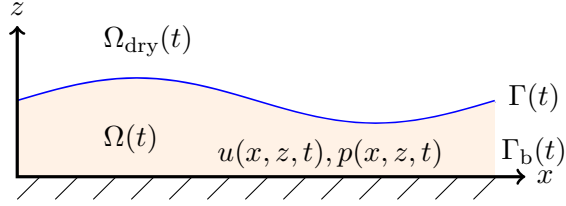

Introducing the potential $\Phi(x,z):=\tfrac{k}{\nu}(-p(x,z)+g\rho z)$, the system \eqref{HSunstable} can be written in terms of $\Phi$ as
\begin{equation}\label{HSunstablePhi}
\begin{cases}\begin{aligned}
&u=\nabla \Phi\qquad &\mbox{in}\quad \Omega(t)\\
&\Delta \Phi=0\qquad &\mbox{in}\quad \Omega(t)\\
&\partial_t \zeta = \nabla \Phi \cdot N\qquad &\mbox{on}\quad \Gamma(t)\\
&\Phi=\tfrac{k}{\nu}\left(g \rho \zeta +\gamma \kappa_{\zeta}\right)\qquad &\mbox{on}\quad \Gamma(t)\\
&\partial_z \Phi=0\qquad &\mbox{on} \quad \Gamma_{\rm bot}  
\end{aligned}
\end{cases}
\end{equation}

Differently from the approach used in \cite{GG-BS2020}, we exploit the divergence-free condition on the fluid velocity to write the kinematic equation in a divergence form. This kind of approach is usually applied when dealing with asymptotic models for the water waves problem \cite{Lan13} and for the floating structures problem \cite{Lan17, Boc20}. Here we adopt it to derive in a more direct way the thin-film approximation of the full one-phase Muskat problem. More precisely, we will write the potential as a formal expansion in terms of some parameters and, after truncating at the leading order, the asymptotic model will be derived.\\
In fact, using the fact that $\Phi$ is a harmonic function and the non-permeability condition at $\Gamma_{\rm bot}$, we obtain that
\begin{align*}
\nabla \Phi \cdot N&= -\partial_x\zeta (\partial_x\Phi )_{|_{z=\zeta}} + (\partial_z \Phi)_{|_{z=\zeta}}  = -\partial_x\zeta (\partial_x\Phi )_{|_{z=\zeta}} + \int_{-H}^{\zeta}\partial_{zz} \Phi dz\\ &=  -\partial_x\zeta (\partial_x\Phi )_{|_{z=\zeta}} -  \int_{-H}^{\zeta}\partial_{xx} \Phi dz= -\partial_x\left(\int_{-H}^{\zeta}\partial_x \Phi dz\right).
\end{align*}
Therefore the one-dimensional one-phase unstable Muskat problem \eqref{HSunstablePhi} reads 
\begin{equation}\label{heleshaw}
\begin{cases}\begin{aligned}
&\partial_t \zeta + \partial_x\left(\int_{-H}^{\zeta}\partial_x \Phi dz\right)=0 \quad&\mbox{in}\quad \mathbb{T} \ \ \  \ \\
&\Delta \Phi=0 &\mbox{in}\quad \Omega(t)\\
&\Phi=\tfrac{k}{\nu}\left(g \rho \zeta +\gamma \kappa_{\zeta}\right)\qquad &\mbox{on}\quad \Gamma(t)\\
&\partial_z \Phi=0\qquad &\mbox{on} \quad \Gamma_{\rm bot} 
\end{aligned}
\end{cases}
\end{equation}

\subsection{Dimensionless one-phase unstable Muskat problem}\label{dimless_muskat}
Let us introduce now three physical parameters intrinsic to the problem:
\begin{itemize}
	\item[-] The characteristic horizontal scale $L$
	\item[-] The characteristic fluid height $H$
	\item[-] The characteristic free surface amplitude $a$
\end{itemize}
These three parameters permit to define other three quantities whose size describes the properties of the fluid motion. They are the \textit{shallowness parameter} $\mu$, the \textit{nonlinearity parameter}  $\eps$ and the \textit{Bond number} $\rm Bo$, defined respectively by 
\begin{equation}\label{parameters}
\mu=\frac{H^2}{L^2}, \qquad \eps=\frac{a}{H}, \qquad \mathrm{Bo}=\frac{ g\rho L^2}{\gamma}.
\end{equation} Notice that the Bond number describes the ratio of gravitational to capillarity forces. We have that $\mu> 0$, $0< \eps \leq 1 $ and $\rm  Bo > 0.$
\begin{remark}\label{remarOsqrtmu}
	As a first approach we consider the case when 
	the typical horizontal and vertical characteristic scales of the problem are such that
$$ \frac{1}{\rm Bo}=O(1).$$ This ceases to hold if the configuration studied is different and involves larger or smaller scales. For instance, the scenario changes in the study of water waves with surface tension, where it is reasonable to assume that (see \cite{Lan13}) $$\frac{1}{\rm Bo}=O(\mu).$$ 
	The different modeling with respect to water waves is relevant because it will allow us to keep the influence of the surface tension when we derive the asymptotic model at order $O(\mu)$. The presence of the surface tension will be crucial to the existence of the solution to the derived thin-film equation, which would be ill-posed in the presence of only gravity due to the unstable configuration we are dealing with.
	\end{remark}

	We then introduce new dimensionless variables 
	\begin{equation}\label{var_dimless}
	 x'=\frac{x}{L}, \quad z'=\frac{z}{H},\quad t'=\frac{t}{t_0}, \quad\mbox{with}\mbox\quad  t_0=\frac{\nu L^2}{k g\rho H},
	\end{equation} and dimensionless unknowns
		\begin{equation}\label{unk_dimless}
	\zeta'=\frac{\zeta}{a}, \qquad \Phi'=\frac{\Phi}{\Phi_0}, \quad\mbox{with}\mbox\quad  \Phi_0=\frac{k g\rho H }{\nu}. 
	\end{equation}

	After omitting the primes and defining the dimensionless domains 
	\begin{align*}
	\Omega^\eps(t)&=\{(x,z)\in \mathbb{T}\times \mathbb{R} \ | \ -1 < z < \eps \zeta(t,x) \},\\
	\Gamma^\eps(t)&=\{(x,\eps\zeta(t,x)) \ | \ x\in \mathbb{T}\},\\
	\Gamma_{\rm bot}&=\{(x,-1) \ | \ x\in \mathbb{T}\},
	\end{align*}we write the dimensionless one-dimensional one-phase unstable Muskat problem :
\begin{equation}\label{heleshaw_dimless}
\begin{cases}\begin{aligned} 
&\partial_t  \zeta + \frac{1}{\eps}\partial_x\left(\int_{-1}^{\eps \zeta}\partial_x \Phi\right)=0 \quad&\mbox{in}\quad \mathbb{T} \ \ \ \ \ \\
&\Delta^\mu \Phi =0&\mbox{in}\quad \Omega^\eps(t)\\
&\Phi=\eps \zeta + \tfrac{\eps}{\rm{Bo}}  \kappa^{\eps,\mu}_{\zeta}\qquad &\mbox{on}\quad\Gamma^\eps(t)\\
&\partial_z \Phi=0\qquad &\mbox{on} \ \quad \Gamma_{\rm bot}
\end{aligned}
\end{cases}
\end{equation}
with the dimensionless Laplace operator $\Delta^\mu= \mu\partial_{xx} + \partial_{zz}$ and the dimensionless surface curvature $$\kappa^{\eps,\mu}_{\zeta}= \frac{\partial_{xx}\zeta}{(1+ (\eps \sqrt{\mu}\partial_x \zeta)^2)^{3/2}}.$$

	\begin{remark}\label{Rcapillary-drivenBond}
	In the purely capillarity-driven scenario, one should replace in \eqref{var_dimless}-\eqref{unk_dimless}
	$$t_0=\frac{k \rho H}{\nu L}, \quad \Phi_0=\frac{\nu L^3}{k \rho H}$$ and 
	define $\rm{Bo}$ by
	$$
	\mathrm{Bo}=\frac{\nu^2 L^3}{k^2 \rho H^2 \gamma}.
	$$
\end{remark} 
Since we are dealing with a free boundary problem, first we transform  the constant coefficients elliptic problem  \eqref{heleshaw_dimless} for the potential $\Phi$ in the moving domain $\Omega^\eps(t)$ into a variable coefficients elliptic problem in a fixed domain $\mathcal{S}$. In particular, we choose the periodic flat strip $\mathcal{S}= \mathbb{T}\times (-1,0)$. To do that, we introduce the diffeomorphism $\Sigma:\mathcal{S}\rightarrow \Omega^\eps(t)$ of the form 
$$\Sigma(t,x,z)=(x, z+\sigma(t,x,z))$$ for some function $\sigma$. We will explain later the choice of this function that it is not unique. After introducing the new potential defined on the flat strip defined by $\phi=\Phi\circ \Sigma$, the evolution equation for $\zeta$ reads
\begin{equation}\label{evoeq_flat}
\partial_t  \zeta +\frac{1}{\eps}\partial_x\left(\int_{-1}^{0}\left((1+\partial_z\sigma)\partial_x \phi-\partial_z\phi \partial_x \sigma\right) dz\right)=0
\end{equation}with $\phi$ satisfying the elliptic problem 
\begin{equation}\label{ellipro_flat}
	\begin{cases}
	\nabla^\mu \cdot P(\Sigma)\nabla^\mu \phi=0 &\mbox{in}\quad \mathcal{S},\\
	\phi=\eps \zeta + \frac{\eps}{\rm{Bo}}\kappa^{\eps,\mu}_{\zeta} & \mbox{on}\quad \mathbb{T}\times\{0\},\\
	\partial_z \phi=0 &\mbox{on}\quad \mathbb{T}\times\{-1\}.
	\end{cases}
\end{equation}with the dimensionless gradient $\nabla^\mu=(\sqrt{\mu}\partial_x, \partial_z)^T$ and $P(\Sigma)=\mathrm{Id} + Q(\Sigma)$ with
\begin{equation}\label{defQ}
	Q(\Sigma)=\begin{pmatrix}
	\partial_z \sigma & -\sqrt{\mu}\partial_x \sigma\\[10pt]
	-\sqrt{\mu}\partial_x \sigma & \dfrac{-\partial_z \sigma +\mu |\partial_x\sigma|^2}{1+\partial_z\sigma}.
	\end{pmatrix}
\end{equation}

\section{Functional framework}\label{fun_sec}
In this section we present the functional framework in which we derive our results for both the thin film equations and the asymptotic stability. Moreover, we show basic properties that we will repeatedly use later on.

Before introducing the functional spaces we are working with, we list some notations that are used throughout the paper.
We consider the torus $\mathbb{T}=\mathbb{R}/2\pi \mathbb{Z}$ that can be thought as the interval $[-\pi,\pi]$ with periodic boundary conditions. Given a function $f\in L^1(\mathbb{T})$, the $n$-th Fourier coefficient of $f$ is given by
$$\widehat{f}(n)=\frac{1}{2\pi}\int_{\mathbb{T}} f(x)e^{-inx}dx,$$ and if $f\in L^2(\mathbb{T})$ it can be written as a Fourier series, that it
$$f(x)=\sum\limits_{n\in\mathbb{Z}} \widehat{f}(n)e^{inx}.$$
Finally, we denote by $\mathcal{S}$ the horizontally periodic strip $\mathbb{T}\times (-1,0).$
\subsection{Wiener spaces on the torus}Let us introduce the Wiener space  $\mathbb{A}^s_\lambda(\mathbb{T})$ with $s,\lambda \geq 0$. It is defined as the space of functions $f\in L^1(\mathbb{T})$ such that the following norm is finite:
$$|f|_{s,\lambda} := \sum\limits_{n\in \mathbb{Z}} (1+|n|)^s e^{\lambda |n|}|\widehat{f}(n)|.$$
Notice that the presence of the exponential weight in the norm $|f|_{s,\lambda}$ implies that, for $\lambda>0$, the Wiener space $\mathbb{A}^s_\lambda(\mathbb{T})$ contains real analytic functions. It is immediate that 
$$\mathbb{A}^{s_1}_\lambda (\mathbb{T})\subset \mathbb{A}^{s_2}_\lambda(\mathbb{T}) \quad \mbox{for}\quad s_1 \geq s_2\geq 0.$$ Furthermore, for zero-mean functions the inhomogeneous norm $|f|_{s,\lambda}$ is equivalent to the homogeneous norm. Indeed, since $\widehat{f}(0)=0$, one has
\begin{equation}\label{equiv-norms}
\tfrac{1}{2^s}|f|_{s,\lambda}\leq \sum\limits_{n\in\mathbb{Z}}|n|^se^{\lambda|n|}|\widehat{f}(n)|\leq  |f|_{s,\lambda}.
\end{equation}
For the sake of simplicity, we write $\mathbb{A}^s_\lambda=\mathbb{A}^s_\lambda(\mathbb{T}).$
We show in the next proposition that Wiener spaces are algebras for $s\geq 0 $. Although it has been already presented in \cite{BG-B2019,GG-BS2020}, we give the detailed proof since we obtain a sharper constant in the product estimate.
\begin{proposition}\label{productWie}
	Let $f,g\in \wiener[s]_\lambda$ with $s,\lambda\geq 0.$ Then,
	\begin{equation}\label{est_product_wie}
	|fg|_{s,\lambda}\leq K_s \left(|f|_{0,\lambda}|g|_{s,\lambda}+ |f|_{s,\lambda}|g|_{0,\lambda}\right)
	\end{equation}
	with 
	$$K_s=\begin{cases}
	1 \quad \mbox{ for} \quad 0<s\leq 1,\\
	2^{s-1} \quad \mbox{for}\quad  s=0, s>1.
	\end{cases}$$
\end{proposition}
\begin{proof}
	From the discrete version of the convolution theorem for the Fourier transform we have that
	\begin{equation*}
	\begin{aligned}
	|fg|_{s,\lambda}=\sum\limits_n(1+|n|)^se^{\lambda |n|} |\widehat{fg}(n)|= \sum\limits_n(1+|n|)^se^{\lambda |n|} | \widehat{f}\ast\widehat{g}(n)|
	\end{aligned}
	\end{equation*}
	and \eqref{est_product_wie} directly follows for $s=0$ using Young's convolution inequality.\\ Let us now consider the case $s\neq 0$.
	On the one hand, taking advantage of the convexity of the function $(1+ x)^s$ for $x\geq0$ and $s>1$, it yields 
	\begin{align*}(1+|n|)^s\leq (1+ \frac{2|n-m|}{2} + \frac{2|m|}{2})^s&\leq\frac{1}{2}\left[(1+ 2|n-m|)^s + (1+2|m|)^s\right]\\&\leq 2^{s-1}\left[(1+ |n-m|)^s + (1+|m|)^s\right].\end{align*}
On the other hand, for $s\in (0,1]$ we use the following inequality 
		\begin{equation*}
			(1+ \alpha)^s \leq 1 + \alpha^s \qquad \forall \alpha\geq 0
		\end{equation*}  to derive that
		\begin{equation}\label{est-s01}
			(x+ y)^s \leq x^s + y^s \qquad \forall x, y\geq 0.
		\end{equation}Then, from the monotonicity of the function $(1 + x)^s$ for positive $s$ and applying \eqref{est-s01} with $x=1+|n-m|$ with $y=1+|m|$, we obtain for $s\in(0,1]$
		\begin{equation*}
				(1+ |n|)^s\leq (1+|n-m|+ 1+ |m|)^s \leq (1+|n-m|)^s + (1+|m|)^s.
		\end{equation*}
	Therefore
	\begin{equation*}
	\begin{aligned}
	&\sum\limits_n(1+|n|)^se^{\lambda |n|} | \widehat{f}\ast\widehat{g}(n)|\leq \sum\limits_{n,m}(1+|n|)^se^{\lambda |n|}  | \widehat{f}(n-m)||\widehat{g}(m)|\\&\leq\sum\limits_{n,m} K_s\left[(1+ |n-m|)^s + (1+|m|)^s\right]e^{\lambda |n-m|} | \widehat{f}(n-m)|e^{\lambda |m|}|\widehat{g}(m)|
	\end{aligned}
	\end{equation*}
	and \eqref{est_product_wie} follows after using Young's convolution inequality.
\end{proof}

Iterating Proposition \ref{productWie}, one can obtain the following estimate for integer powers of a function in a Wiener space:
\begin{proposition}\label{powerWie}
	Let $f\in \mathbb{A}^s_\lambda$ with $s>0$ and $\lambda\geq 0$. Then for $n\in\mathbb{N}$ with $n\geq 2$
	\begin{equation}\label{est_wie_power}
	|f^n|_{s,\lambda} \leq  K_{s,n}|f|^{n-1}_{0,\lambda}|f|_{s,\lambda}
	\end{equation}with 
	$$K_{s,n}=\begin{cases}
	n \quad \mbox{ for}\quad 0< s\leq 1,\\[10pt]
	\sum\limits_{j=1}^{n-1}(2K_s)^j= \frac{2K_s((2K_s)^{n-1}-1)}{2K_s-1} \quad \mbox{ for}\quad  s>1.
	\end{cases}$$
	For $s=0$, 
	\begin{equation}\label{est_wie0_power}
	|f^n|_{0,\lambda} \leq |f|^{n}_{0,\lambda}.
	\end{equation}with 
\end{proposition}

Moreover, the following interpolation inequality holds and we refer to \cite{BG-B2019} for the proof.
\begin{proposition}\label{interpolation}
	Let $0\leq s_1\leq s_2$ and $\lambda\geq 0$. If $f\in \wiener[s_1]_\lambda$ and $f\in\wiener[s_2]_\lambda$, then for any $\theta\in[0,1]$ and $s_\theta:=\theta s_1 + (1-\theta)s_2$ one has
	\begin{equation}
	|f|_{s_\theta,\lambda}\leq |f|_{s_1,\lambda}^\theta|f|_{s_2,\lambda}^{1-\theta}.
	\end{equation}\
\end{proposition}

We will need in Section \ref{ellipest_sec} a control on the Wiener norm of a function that is a composition of a function in a Wiener space with a particular analytic function. We show here a regularity estimate that will be used later to obtain the desired control.

\begin{lemma}\label{estimate_G}
Let $G$ be defined by $G(x)=\frac{1}{(1+x)^{3/2}}-1$.
		 If $v\in \mathbb{A}^s_\lambda$ with $s,\lambda\geq 0$ and $4K_s|v|_{0,\lambda}<1$, then 
	\begin{equation}\label{est_G}|G(v)|_{s,\lambda}\leq 18K_s |v|_{s,\lambda}.
	\end{equation}
\end{lemma}
\begin{proof}
Since $G$ is real analytic and $G(0)=0$, we know that for $|x|<1$
\begin{equation}\label{analyticx}
G(x)=\sum_{n\geq 1} \frac{G^{(n)}(0)}{n!} x^n\end{equation} and from the explicit expression of $G$ also that 
\begin{equation}\label{analytic-x}
G(-x)= \sum_{n\geq 1} \frac{|G^{(n)}(0)|}{n!} x^n.
\end{equation}
Therefore if $|v(x)|<1$ we have
$$|G(v(x))|\leq \sum_{n\geq 1} \frac{|G^{(n)}(0)|}{n!} |v^n(x)|.$$
Using that for $s\geq 0$
$$|v^n(x)|\leq |v^n|_{L^\infty}\leq |v^n|_{0,\lambda}\leq |v^n|_{s,\lambda}$$
and applying Proposition \ref{powerWie} it yields 
$$|G(v(x))|\leq |v|_{s,\lambda} \sum_{n\geq 1} \frac{|G^{(n)}(0)|}{n!} K_{s,n}|v|^{n-1}_{0,\lambda}.$$
We distinguish now when $0\leq s\leq 1$ and $s>1$. In the first case, $K_{s,n}\leq n$
and since \eqref{analytic-x} implies
$$\frac{d}{dx} (G(-x))= \sum_{n\geq 1} \frac{|G^{(n)}(0)|}{n!} n x^{n-1},$$
then we have
\begin{equation*}
\sum_{n\geq 1} \frac{|G^{(n)}(0)|}{n!} K_{s,n}|v|^{n-1}_{0,\lambda}\leq \sum_{n\geq 1} \frac{|G^{(n)}(0)|}{n!} n|v|^{n-1}_{0,\lambda}= \frac{3}{2(1-|v|_{0,\lambda})^{5/2}}
\end{equation*}if $|v|_{0,\lambda}<1$. In the second case, $K_{s,n}\leq (2K_s)^n=$
 and we get 
 \begin{equation*}\begin{aligned}
 	\sum_{n\geq 1} \frac{|G^{(n)}(0)|}{n!} K_{s,n}|v|^{n-1}_{0,\lambda}&\leq 2K_s\sum_{n\geq 1} \frac{|G^{(n)}(0)|}{n!} (2K_s|v|_{0,\lambda})^{n-1}\\[5pt]&
 	\leq 2K_s\sum_{n\geq 1} \frac{|G^{(n)}(0)|}{n!} n (2K_s|v|_{0,\lambda})^{n-1} \leq \frac{ 3K_s}{(1-2K_s|v|_{0,\lambda})^{5/2}}
 	\end{aligned}
 \end{equation*}if $2K_s|v|_{0,\lambda}<1$. Summing up, we have for $s\geq 0$
 \begin{equation*}
 	|G(v(x))|\leq \frac{ 3K_s}{(1-2K_s|v|_{0,\lambda})^{5/2}} |v|_{s,\lambda}
 \end{equation*} and starting from \eqref{analyticx} we can obtain in an analogous way
 \begin{equation*}
 |G(v)|_{s,\lambda}\leq \frac{ 3K_s}{(1-2K_s|v|_{0,\lambda})^{5/2}} |v|_{s,\lambda}, 
\end{equation*} which implies \eqref{est_G} if $4K_s|v|_{0,\lambda}<1.$
\end{proof}

\subsection{Wiener-Sobolev spaces in the horizontal strip}
 
	For $s,\lambda\geq 0$ and $k\in \mathbb{N}$ we consider the anisotropic Wiener-Sobolev spaces $\mathcal{A}^{s,k}_\lambda(\mathcal{S})$. It is defined as the space of functions $f\in L^1(\mathbb{T}\times (-1,0))$ such that $\partial_z^k \hat{f}(n,\cdot)\in L^1(-1,0) $ for any $n\in \mathbb{Z}$ and the following norm is finite:
$$\|f\|_{\mathcal{A}^{s,k}_\lambda}= \sum\limits_{n\in\mathbb{Z}}(1+|n|)^s e^{\lambda|n|}\int_{-1}^0 |\partial_z^k \widehat{f}(n,z)|dz.$$
For the sake of simplicity, we write $\mathcal{A}^{s,k}_\lambda=\mathcal{A}^{s,k}_\lambda(\mathcal{S}).$ 
Analogously to Wiener spaces, we can derive a product estimate for $\mathcal{A}^{s,0}_\lambda$.
The proof of the following proposition is based on Proposition \ref{productWie} and Proposition \ref{powerWie} and we refer to Lemma 1.6 in \cite{GG-BS2020} for more details. 

\begin{proposition}\label{productWieSob}
	Let $f \in \mathcal{A}^{s,0}_\lambda $ and $g\in \mathcal{A}^{s,1}_\lambda$  with $s,\lambda\geq 0.$ Then $fg \in \mathcal{A}^{s,0}_\lambda$ and
	\begin{equation*}\label{product_wiesob}
	\|fg\|_{\mathcal{A}^{s,0}_\lambda} \leq K_s\left(\|f\|_{\mathcal{A}^{s,0}_\lambda} \|g\|_{\mathcal{A}^{0,1}_\lambda} + \|f\|_{\mathcal{A}^{0,0}_\lambda} \|g\|_{\mathcal{A}^{s,1}_\lambda}\right)
	\end{equation*}
	with $K_s$ as in Proposition \ref{productWie}.
\end{proposition}

\section{Lubrication  approximations}\label{lubri_sec}
In this section we consider lubrication approximations of the dimensionless one-phase Muskat problem, which are widely used in applications when a particular regime occurs. This is the so-called \textit{thin film} regime, where the fluid height is small compared to the horizontal length scale of the problem. With our notation, this means that $$\sqrt{\mu}=\frac{H}{L} \ll 1.$$We will consider different approximations of the problem in the thin-film regime, truncating the equations at different orders.\\\\
Let us write the flattened one-phase Muskat problem \eqref{evoeq_flat}-\eqref{ellipro_flat} in a different way. 
From \cite{Lan13}, we know that the elliptic operator $	\nabla^\mu \cdot P(\Sigma)\nabla^\mu$ can be written as the sum of two differential operators, the first one independent of $\mu$ and the second one dependent of $\mu$, namely
$$\nabla^\mu \cdot P(\Sigma)\nabla^\mu = \tfrac{1}{1+\eps\zeta} \partial_{zz} + \mu A_\sigma(\partial_x,\partial_z)$$ with 
$$A_\sigma(\partial_x, \partial_z) \cdot = \partial_x((1+\eps\zeta)\partial_x \cdot )  - \partial_z(\partial_x\sigma \partial_x \cdot)+ \partial_z\left(\frac{(\partial_x \sigma)^2}{1+\eps\zeta}\partial_z \cdot\right)  - \partial_x(\partial_x\sigma \partial_z \cdot).$$
Secondly, we write the dimensionless surface curvature as the sum of its linearized part and the remainder, that is
\begin{equation}\label{defcurv_dimless} \kappa^{\eps,\mu}_{\zeta}= \partial_{xx} \zeta+F_{\eps\sqrt{\mu}}(\partial_x \zeta)\partial_{xx} \zeta, \end{equation} where $F_{\eps\sqrt{\mu}}(\partial_x \zeta)=F(\eps\sqrt{\mu}\partial_x\zeta) $ with $$F(x)=\dfrac{1}{(1+x^2)^{3/2}}-1.$$ From the Taylor development of $F$ at $x=0$, we observe that $F_{\eps\sqrt{\mu}}(\partial_x \zeta)$ is $O(\mu)$.\\
We introduce now the following ansatz: we look for the potential $\phi$ as an asymptotic expansion in terms of $\sqrt{\mu}$, \textit{i.e.}  \begin{equation}\label{ansatz}\phi ^\mu=\phi^0 +\sqrt{\mu} \phi^{1/2}+ \widetilde{\phi}^\mu\end{equation} with $\phi^0$ satisfying
\begin{equation}\label{laplace_phi0}
\begin{cases}
\frac{1}{1+\eps\zeta}\partial_{zz} \phi^0=0&\mbox{in}\quad \mathcal{S},\\[5pt]
\phi^0=\eps (\zeta +  \tfrac{1}{\rm Bo}\partial_{xx}\zeta) & \mbox{on}\quad \mathbb{T}\times\{0\},\\[5pt]
\frac{1}{1+\eps\zeta}\partial_z \phi^0=0&\mbox{on}\quad \mathbb{T}\times\{-1\},
\end{cases}
\end{equation}$\phi^{1/2}$ satisfying 
\begin{equation}\label{laplace_phi1/2}
\begin{cases}
\frac{1}{1+\eps\zeta}\partial_{zz} \phi^{1/2}=0&\mbox{in}\quad \mathcal{S},\\[5pt]
\phi^{1/2}=0 & \mbox{on}\quad \mathbb{T}\times\{0\},\\[5pt]
\frac{1}{1+\eps\zeta}\partial_z \phi^{1/2}=0&\mbox{on}\quad \mathbb{T}\times\{-1\}.
\end{cases}
\end{equation} and the remainder potential $\widetilde{\phi}^\mu$ satisfying
\begin{equation}\label{laplace_phimu}
\begin{cases}
\nabla^\mu \cdot P(\Sigma)\nabla^\mu \widetilde{\phi}^\mu=-\mu A_\sigma(\partial_x,\partial_z)(\phi^0+\sqrt{\mu}\phi^{1/2})&\mbox{in}\quad \mathcal{S},\\[5pt]
\widetilde{\phi}^\mu= \frac{\eps}{\mathrm{Bo}} F_{\eps\sqrt{\mu}}(\partial_x \zeta)\partial_{xx} \zeta & \mbox{on}\quad \mathbb{T}\times\{0\},\\[5pt]
\frac{1}{1+\eps\zeta}\partial_z \widetilde{\phi}^\mu=0&\mbox{on}\quad \mathbb{T}\times\{-1\}.
\end{cases}
\end{equation}
It is straightforward that systems \eqref{laplace_phi0} and \eqref{laplace_phi1/2} admit respectively the unique explicit solutions 
\begin{equation}\label{phi0}
\begin{aligned}
\phi^0(x,z)= \phi^0(x)=\eps (\zeta+ \tfrac{1}{\rm{Bo}}  \partial_{xx} \zeta)
\end{aligned}
\end{equation}
and 
\begin{equation}\label{phi1/2}
\begin{aligned}
\phi^{1/2}(x,z)\equiv 0.
\end{aligned}
\end{equation}
Let us now choose the diffeomorphism $\Sigma$ such that $\Sigma^{-1}$ flattens the fluid domain $\Omega^\eps(t)$ into the horizontal strip $\mathcal{S}$. We consider $\Sigma(x,z)=(x, z+\sigma(x,z))$ with 
\begin{equation}\label{sigma}
\sigma(x,z)=\eps\zeta(x)(z+1).
\end{equation}Notice that \eqref{sigma} is not the unique possible choice for $\sigma$. For instance, one can define $\sigma$ as the harmonic extension of the free surface $\zeta$ (as done in \cite{GG-BS2020} for small perturbations of flat interfaces) or as a pseudo-differential vertical localization (as presented in \cite{Lan13}). We will see that, although with our simple choice of $\sigma$ the diffeomorphism $\Sigma$ is not \textit{regularizing} (in the sense of \cite{Lan13}), we get the same control for $\nabla \sigma$ in the Wiener-Sobolev space $\mathcal{A}^{s,0}_\lambda$ as in \cite{GG-BS2020} in terms of regularity of the free surface $\zeta.$ With this choice, $A_\sigma(\partial_x,\partial_z)$ reduces to
\begin{equation*}
\begin{aligned}
A_\zeta(\partial_x,\partial_z) \cdot = \ &\partial_x((1+\eps\zeta)\partial_x \cdot )  - \partial_z(\eps(z+1)\partial_x\zeta \partial_x \cdot)\\[5pt]&+ \partial_z\left(\frac{(\eps(z+1)\partial_x \zeta)^2}{1+\eps\zeta}\partial_z \cdot\right)  - \partial_x(\eps(z+1)\partial_x\zeta \partial_z \cdot).
\end{aligned}
\end{equation*}
 In particular, when $A_\zeta(\partial_x,\partial_z)$ is applied to a function independent of the vertical variable $z$ it reads 	$A_\zeta(\partial_x,\partial_z) =(1+\eps\zeta)\partial_{xx} $.\\
 Therefore, after injecting \eqref{sigma}, \eqref{ansatz}, \eqref{phi0} and \eqref{phi1/2} into \eqref{evoeq_flat}, the one-dimensional one-phase unstable Muskat problem reduces to the evolution equation
\begin{equation}\label{HSeq_dimless}
\begin{aligned}
\partial_t \zeta +& \partial_x\left((1+\eps \zeta)\left(\partial_x \zeta + \tfrac{1}{\rm{Bo}}\partial_{xxx} \zeta\right)\right)\\&= -\tfrac{1}{\eps}\partial_x\left(\int_{-1}^0\left((1+\eps\zeta)\partial_x \widetilde{\phi}^\mu - \eps(z+1)\partial_x\zeta \partial_z\widetilde{\phi}^\mu\right)dz\right)
\end{aligned}
\end{equation} with $\widetilde{\phi}^\mu$ satisfying the elliptic problem \eqref{laplace_phimu}.

\subsection{Approximation at order $O(\mu)$}Let us truncate \eqref{HSeq_dimless} by dropping the terms of order $O(\mu)$. Therefore, we obtain the following non-degenerate fourth-order parabolic equation 
\begin{equation}\label{thin-film}
\partial_t \zeta + \partial_x\left((1+\eps \zeta)(\partial_x \zeta + \tfrac{1}{\rm{Bo}}\partial_{xxx} \zeta)\right)=0.\vspace{1em}
\end{equation} 

This equation is a generalization of the well-known thin-film equation that takes into account the influence of both gravity and surface tension, the latter appearing in the Bond number $\textrm{Bo}$ (see definition in \eqref{parameters}). In some sense, the capillary effects balance the gravitational effects and allow to obtain existence of the solution to the approximated equation also in the unstable case.
Let us now introduce the definition of weak solution to \eqref{thin-film}.

\begin{definition}\label{weakdefTF}
We say that $\zeta\in L^\infty([0,T],L^\infty(\mathbb{T}))\cap L^1([0,T],W^{3,\infty}(\mathbb{T})) $ is a weak solution to \eqref{thin-film} with initial data $\zeta_0\in L^\infty(\mathbb{T})$ if for any $\varphi\in C^\infty([0,T]\times \mathbb{T})$ and $t\in[0,T]$ the following equality holds:
\begin{equation}
	\begin{aligned}
	&\int_\mathbb{T} \zeta(t,x)\varphi(t,x) dx-	\int_\mathbb{T} \zeta_0(x)\varphi(0,x)dx - \int_0^t\int_\mathbb{T} \zeta(t')\partial_t \varphi(t',x)dx dt'\\
	&\quad=\int_0^t\int_\mathbb{T}(1+\eps \zeta(t,x)) \partial_x  \zeta(t',x)\partial_x\varphi(t',x)dxdt'\\&\quad\quad +\frac{1}{\rm Bo}\int_0^t\int_\mathbb{T}(1+\eps \zeta(t,x)) \partial_{xxx} \zeta(t',x)\partial_x\varphi(t',x)dxdt'.
	\end{aligned}
	\label{weakTF}
\end{equation}
\end{definition}

We are now able to show the existence of the solution to \eqref{thin-film} in Wiener spaces.

\begin{theorem}\label{theoTF}
Let $0<\rm Bo < 1$ and $\zeta_0\in \wiener[0]_0(\mathbb{T})$ a zero-mean function such that 
\begin{equation}\label{smallnessTF}
|\zeta_0|_{0,0}< \tfrac{\tfrac{1}{\mathrm{Bo}}-1}{8\eps(\tfrac{1}{\mathrm{Bo}}+1)}.\end{equation}
Then there exists a global weak solution to \eqref{thin-film} in the sense of Definition \ref{weakdefTF} with initial data $\zeta_0$ which becomes instantaneously analytic in a growing strip in the complex plane. In particular, for $$\nu\in \left[0,\tfrac{1}{2}(\tfrac{1}{\mathrm{Bo}}-1)\right),$$ we have  $$\zeta\in L^\infty([0,T];\wiener[0]_{\nu t})\cap L^1([0,T];\wiener[4]_{\nu t})$$
for all $T>0$ and the solution $\zeta$ satisfies for any $t\in[0,T]$ the energy inequality
\begin{equation}\label{energy_estimate}
	|\zeta(t)|_{0,\nu t} + \tfrac{1}{64} \left(\tfrac{1}{\mathrm{Bo}}-1\right)\int_0^t|\zeta(t')|_{4,\nu t'}dt'\leq	|\zeta_0|_{0,0}
\end{equation} together with the decay
\begin{equation}\label{decay}
|\zeta(t)|_{0,\nu t} \leq 	|\zeta_0|_{0,0} \ e^{-\frac{t}{64}\left(\tfrac{1}{\mathrm{Bo}}-1\right)}.
\end{equation}
\end{theorem}

\begin{proof}
In order to get the existence of the solution to \eqref{thin-film}, first we need  to derive a priori parabolic estimates.  We rewrite \eqref{thin-film} as 
\begin{equation}\begin{aligned}
	\partial_t\zeta  +&\tfrac{1}{\rm Bo}\partial_{xxxx}\zeta +\partial_{xx}\zeta\\[5pt]& = -\tfrac{\eps}{\mathrm{Bo}}(\partial_x\zeta \partial_{xxx}\zeta + \zeta \partial_{xxxx}\zeta) - \eps(\partial_x\zeta)^2 - \eps\zeta\partial_{xx}\zeta = \rm NL_2(\zeta)
	\end{aligned}
\end{equation}where we have denoted by $\rm NL_2(\zeta)$ the quadratic nonlinear terms of the equation.
Considering the Fourier series of \eqref{thin-film}  we obtain 
\begin{equation*}
\partial_t \widehat{\zeta} + \tfrac{1}{\mathrm{Bo}}|n|^4 \widehat{\zeta} - |n|^2 \widehat{\zeta} = \widehat{\rm NL_2(\zeta)},
\end{equation*}where $\widehat{\zeta}$ denotes the Fourier coefficient of $\zeta$.
Multiplying by the complex conjugate $\overline{\widehat{\zeta}}$, dividing by $|\widehat{\zeta}|$ and taking the real part  it yields 
\begin{equation}\label{fourier-est}
	\partial_t |\widehat{\zeta}| + \tfrac{1}{\mathrm{Bo}}|n|^4 |\widehat{\zeta}| - |n|^2 |\widehat{\zeta}| = \tfrac{\rm Re(\widehat{NL_2(\zeta)}\ \overline{\widehat{\zeta}}) }{|\widehat{\zeta}|} \leq |\widehat{\rm NL_2(\zeta)}|,
\end{equation} Recalling that 
\begin{equation*}
		\frac{d}{dt} | \zeta|_{s,\nu t} = \sum_n(1+|n|)^s e^{\nu t|n|}\partial_t |\widehat{\zeta}| + \nu |\Lambda \zeta|_{s,\nu t} ,
\end{equation*}
we multiply \eqref{fourier-est} by $e^{\nu t |n|}$ and taking the sum for $n\in \mathbb{Z}$ we obtain for $s=0$
\begin{equation*}
	\frac{d}{dt} | \zeta|_{0,\nu t}  + (\tfrac{1}{\mathrm{Bo}} -1-\nu) \sum_n |n|^4 e^{\nu t|n|}   |\widehat{\zeta}| \leq |\rm NL_2(\zeta)|_{0,\nu t}.
\end{equation*}
Choosing $\nu\in \left[0,\tfrac{1}{2}(\tfrac{1}{\mathrm{Bo}}-1)\right)$ and due to \eqref{equiv-norms} it yields
\begin{equation}
\frac{d}{dt} | \zeta|_{0,\nu t} +\tfrac{1}{32}(\tfrac{1}{\mathrm{Bo}}-1)|\zeta|_{4,\nu t}\leq |\rm NL_2(\zeta)|_{0,\nu t}.
\label{apriori}
\end{equation}
Let us now estimate each term in the right-hand side of \eqref{apriori}. By repetitively using Proposition  \ref{productWie} and \ref{interpolation}, we have 
\begin{equation*}
	\begin{aligned}
	|\partial_x \zeta \partial_{xxx}\zeta|_{0,\nu t } &\leq |\partial_x\zeta|_{0,\nu t}|\partial_{xxx}\zeta|_{0,\nu t } \leq |\zeta|_{1,\nu t}|\zeta|_{3,\nu t}  \leq  |\zeta|_{0,\nu t} |\zeta|_{4,\nu t},
	\end{aligned}
\end{equation*}
\begin{equation*}
\begin{aligned}
|\zeta \partial_{xxxx}\zeta|_{0,\nu t} \leq|\zeta|_{0,\nu t}|\partial_{xxxx}\zeta|_{0,\nu t} \leq  |\zeta|_{0,\nu t} |\zeta|_{4,\nu t},
\end{aligned}
\end{equation*}
\begin{equation*}
\begin{aligned}
|(\partial_x\zeta)^2|_{0,\nu t} \leq  |\partial_x\zeta|^2_{0,\nu t}\leq |\zeta|^2_{1,\nu t}\leq|\zeta|_{0,\nu t} |\zeta|_{2,\nu t},
\end{aligned}
\end{equation*}
\begin{equation*}
\begin{aligned}
|\zeta \partial_{xx}\zeta|_{0,\nu t} &\leq  |\zeta|_{0,\nu t}|\partial_{xx}\zeta|_{0,\nu t}\leq |\zeta|_{0,\nu t} |\zeta|_{2,\nu t}.
\end{aligned}
\end{equation*}
Gathering all the previous estimates together, from \eqref{apriori} we get 
\begin{equation*}
	\frac{d}{dt} | \zeta|_{0,\nu t} +\tfrac{1}{32}(\tfrac{1}{\mathrm{Bo}}-1)|\zeta|_{4,\nu t}\leq  2\eps (\tfrac{1}{\mathrm{Bo}}+1)|\zeta|_{0,\nu t} |\zeta|_{4,\nu t}
\end{equation*}
and, assuming $|\zeta|_{0,\nu t} < \tfrac{\tfrac{1}{\mathrm{Bo}}-1}{128\eps \left(\tfrac{1}{\mathrm{Bo}}+1\right)},$ we derive the following inequality:
\begin{equation}
\frac{d}{dt} | \zeta|_{0,\nu t} +\tfrac{1}{64}(\tfrac{1}{\mathrm{Bo}}-1)|\zeta|_{4,\nu t}\leq  0. \label{aprioriest}
\end{equation}From the smallness assumption \eqref{smallnessTF}, one has
$$\|\zeta\|_{L^\infty(\mathbb{T})}\leq |\zeta|_{0,0}\leq |\zeta|_{0,\nu t}< \tfrac{\tfrac{1}{\mathrm{Bo}}-1}{128\eps (\tfrac{1}{\mathrm{Bo}}+1)}\leq \tfrac{1}{128 \eps},$$
which implies $1+\eps\zeta(t) >1-\tfrac{1}{128} >0$ for all $t$. This means that no pinch-off occurs and the estimate \eqref{aprioriest} is global in time.\\
Using a standard approximation argument together with the a priori estimate \eqref{aprioriest}, one can show the existence of the weak solution to \eqref{thin-film}, which satisfies \eqref{energy_estimate}. We omit the details of the approximation argument since it is not the aim of this paper, and we refer to \cite{GG-BS2020} the interested reader. Moreover, since $\mathbb{A}^4_{\nu t} \subset \mathbb{A}^{0}_{\nu t}$, from \eqref{energy_estimate} we obtain 
\begin{equation*}
|\zeta(t)|_{0,\nu t} \leq - \tfrac{1}{64} \left(\tfrac{1}{\mathrm{Bo}}-1\right)\int_0^t|\zeta(t')|_{0,\nu t}dt' +	|\zeta_0|_{0,0}
\end{equation*}
and using Gronwall's inequality the decay \eqref{decay} follows.
\end{proof}

\subsection{Approximation at order $O(\mu^{3/2})$}
Motivated by the aim to obtain new thin-film models that can handle the unstable configuration of the one-phase Muskat problem, we want to go further in the expansion of the potential $\phi^\mu$ with respect to  powers of the parameter $\sqrt{\mu}$ . More precisely, we refine \eqref{ansatz} by writing 
\begin{equation}\label{ansatz-refined}\phi^\mu=\phi^0 + \mu \phi^1+ \widetilde{\phi}^\mu_{\rm ref}\end{equation} with $\phi^0= \eps (\zeta+ \tfrac{1}{\rm{Bo}}  \partial_{xx} \zeta)$, $\phi^1$ satisfying 

\begin{equation}\label{laplace_phi1}
\begin{cases}
\frac{1}{1+\eps \zeta} \partial_{zz}\phi^1= -A_\zeta(\partial_x,\partial_z)\phi^0&\mbox{in}\quad \mathcal{S},\\[5pt]
 \phi^1= -\tfrac{3 \eps^3}{2 \rm Bo}(\partial_x \zeta)^2\partial_{xx}\zeta& \mbox{on}\quad \mathbb{T}\times\{0\},\\[5pt]
\frac{1}{1+\eps\zeta}\partial_z \phi^1=0&\mbox{on}\quad \mathbb{T}\times\{-1\},
\end{cases}
\end{equation}
 and now the refined remainder potential 
$\widetilde{\phi}^\mu_{\rm ref}$ satisfies

\begin{equation}\label{laplace_phimu_ref_ill}
\begin{cases}
\nabla^\mu \cdot P(\Sigma)\nabla^\mu \widetilde{\phi}^\mu_{\rm ref}=-\mu^2 A_\zeta(\partial_x,\partial_z)\phi^1&\mbox{in}\quad \mathcal{S},
\\[5pt]
\widetilde{\phi}^\mu_{\rm ref}= \tfrac{\eps}{\mathrm{Bo}}( F_{\eps\sqrt{\mu}}(\partial_x \zeta) +\tfrac{3}{2}\eps^2\mu (\partial_x\zeta)^2)\partial_{xx} \zeta 
& \mbox{on}\quad \mathbb{T}\times\{0\},\\[5pt]
\frac{1}{1+\eps\zeta}\partial_z \widetilde{\phi}^\mu_{\rm ref}=0&\mbox{on}\quad \mathbb{T}\times\{-1\},
\end{cases}
\end{equation}Notice that in \eqref{ansatz-refined} we have used that $\phi^{1/2}$ identically vanishes. Recalling the definition of $A_\zeta(\partial_x, \partial_z)$, the source term of the elliptic problem \eqref{laplace_phi1} reads
\begin{equation}\label{Asigma}\begin{aligned}
	A_\zeta(\partial_x,\partial_z)\phi^0=A_\zeta(\partial_x,\partial_z)(\eps(\zeta+\tfrac{1}{\rm Bo}\partial_{xx}\zeta))=\eps(1+\eps\zeta)(\partial_{xx}\zeta + \tfrac{1}{\rm Bo}\partial_{xxxx}\zeta)
	\end{aligned}
\end{equation}
after having used \eqref{phi0}.
Therefore, the potential $\phi^1$ can be explicitly determined and it reads 
\begin{equation}\label{phi1}
\begin{aligned}
\phi^1(x,z)=-\eps\left(\tfrac{z^2}{2}+z\right)(1+\eps \zeta)^2 (\partial_{xx}\zeta + \tfrac{1}{\rm Bo}\partial_{xxxx}\zeta) -\tfrac{3 \eps^3}{2 \rm Bo}(\partial_x \zeta)^2\partial_{xx}\zeta.
\end{aligned}
\end{equation}
We inject the refined ansatz \eqref{ansatz-refined} into \eqref{evoeq_flat} together with \eqref{phi0}, \eqref{phi1/2} and \eqref{phi1} to recast the one-phase unstable Muskat problem as
\begin{equation}\label{HSeq_dimless2}
\begin{aligned}
&\partial_t  \zeta + \partial_x\left((1+\eps \zeta)(\partial_x \zeta + \tfrac{1}{\mathrm{Bo}}\partial_{xxx} \zeta)\right)  \\[5pt]&+ \tfrac{\mu}{3}\partial_{xx}\left((1+\eps\zeta)^3 (\partial_{xx}\zeta+ \tfrac{1}{\rm Bo}\partial_{xxxx}\zeta)\right) -\tfrac{3\eps^2\mu}{2\rm Bo}\partial_x \left((1+\eps\zeta) \partial_x ((\partial_x\zeta)^2\partial_{xx}\zeta)\right) \\[5pt]&=-\tfrac{1}{\eps}\partial_x\left(\int_{-1}^0\left((1+\eps\zeta)\partial_x \widetilde{\phi}^\mu_{\rm ref} - \eps(z+1)\partial_x\zeta \partial_z\widetilde{\phi}^\mu_{\rm ref}\right)dz\right)
\end{aligned}
\end{equation} with $\widetilde{\phi}^\mu$ satisfying the elliptic problem \eqref{laplace_phimu_ref_ill}. From the definition of the remainder potential, we know that $\widetilde{\phi}^\mu$ is of order $O(\mu^{3/2})$. Therefore, taking into account the terms of order $O(\mu)$ and dropping the terms of order $O(\mu^{3/2})$, we obtain the following sixth-order thin film approximation:
\begin{equation}\label{O(mu)_ill}
\begin{aligned}
&\partial_t  \zeta + \partial_x\left((1+\eps \zeta)(\partial_x \zeta + \tfrac{1}{\mathrm{Bo}}\partial_{xxx} \zeta)\right)\\[5pt]& + \tfrac{\mu}{3}\partial_{xx}\left((1+\eps\zeta)^3 (\partial_{xx}\zeta+ \tfrac{1}{\rm Bo}\partial_{xxxx}\zeta)\right) -\tfrac{3\eps^2\mu}{2\rm Bo}\partial_x \left((1+\eps\zeta) \partial_x ((\partial_x\zeta)^2\partial_{xx}\zeta)\right) =0.
\end{aligned}
\end{equation}
 It is quite surprising that the thin-film approximation at this order leads to an ill-posed sixth-order equation. Indeed, one can notice that at the linear level \eqref{O(mu)_ill} reads
	$$\partial_t \zeta + \tfrac{\mu}{3\rm Bo} \partial_{xxxxxx}\zeta + \mbox{low order terms}=0$$ whose leading order term has a bad sign. Indeed, in the Fourier formulation it gives
	$$\partial_t \widehat{\zeta} = \tfrac{\mu}{3\rm Bo}|n|^6 \widehat{\zeta}+ \mbox{low order terms}$$ and the solution does not belong to any Sobolev or Wiener space.
	\begin{remark}
		The leading order term in \eqref{O(mu)_ill} is due to the surface tension term in the boundary condition of the Hele-Shaw problem and not from the gravity term. This means that one faces the same bad sign either in the stable or in the unstable case using this second-order thin-film approximation.
	\end{remark}
	At a first look, the bad sign comes surprisingly since one would expect the presence of parabolicity in the equation after a second-order approximation as it occurs in the full unstable Hele-Shaw problem studied in \cite{GG-BS2020}. However it is known that, depending on the type of approximation one considers, high order truncations can lead to ill-posed asymptotic models. For instance, this is the case of some truncated series models for the water waves problem discussed in \cite{AmbroBonaNich2014}.\\
	The obstacle represented by the surface tension term motivates us to change the physical regime of the problem we are studying. More precisely, we consider now the case when 
	$$\frac{1}{\rm Bo}=\frac{\sqrt{\mu}}{\rm bo}=O(\sqrt{\mu})$$with the rescaled Bond number $$\mathrm{bo}=\frac{\rho g H L}{\gamma}.$$ 
	\begin{remark}
We highlight that this configuration is still physically reasonable when dealing with groundwater filtration. For instance, considering the characteristic scales $H\sim 0.1 - 1$ mm and $L\sim 5$ mm, the situation studied would belong to this regime. 
	\end{remark}
In this new regime we write the potential $\phi^\mu$ as
\begin{equation}\label{ansatz-refined-new}
	\phi^\mu=\phi^0 + \sqrt{\mu}\phi^{1/2} + \mu \phi^1 + \widetilde{\phi}^\mu_{\rm ref}
\end{equation}where now $\phi^0= \eps\zeta$, $\phi^{1/2}= \tfrac{\eps}{\rm bo}\partial_{xx}\zeta,$ $ \phi^1$ satisfies
\begin{equation}\label{laplace_phi1new}
\begin{cases}
\frac{1}{1+\eps \zeta} \partial_{zz}\phi^1= -A_\zeta(\partial_x,\partial_z)\phi^0&\mbox{in}\quad \mathcal{S},\\[5pt]
\phi^1= 0& \mbox{on}\quad \mathbb{T}\times\{0\},\\[5pt]
\frac{1}{1+\eps\zeta}\partial_z \phi^1=0&\mbox{on}\quad \mathbb{T}\times\{-1\},
\end{cases}
\end{equation}and the refined remainder potential 
$\widetilde{\phi}^\mu_{\rm ref}$ satisfies
\begin{equation}\label{laplace_phimu_ref}
\begin{cases}
\nabla^\mu \cdot P(\Sigma)\nabla^\mu \widetilde{\phi}^\mu_{\rm ref}=-\mu^{3/2}A_\zeta(\partial_x, \partial_z)\phi^{1/2}-\mu^2 A_\zeta(\partial_x,\partial_z)\phi^1&\mbox{in}\,\, \mathcal{S},
\\[5pt]
\widetilde{\phi}^\mu_{\rm ref}= \tfrac{\eps\sqrt{\mu}}{\mathrm{bo}} F_{\eps\sqrt{\mu}}(\partial_x \zeta) \partial_{xx} \zeta 
& \mbox{on}\,\, \mathbb{T}\!\times\!\{0\},\\[5pt]
\frac{1}{1+\eps\zeta}\partial_z \widetilde{\phi}^\mu_{\rm ref}=0&\mbox{on}\,\, \mathbb{T}\!\times\!\{-1\}.
\end{cases}
\end{equation}The explicit expressions of $\phi^0$ and $\phi^{1/2}$ are obtained by identifying the terms of the same order in the expansion of the original elliptic problem on $\phi^\mu$. This leads to systems analogous to \eqref{laplace_phi0}-\eqref{laplace_phi1/2} that admit explicit solutions. Moreover, it is straightforward that
\eqref{laplace_phi1new}  admits the unique explicit solution
$$\phi_1(x,z)=-\eps\left(\tfrac{z^2}{2}+z\right)(1+\eps\zeta)^2\partial_{xx}\zeta.$$
Injecting the ansatz \eqref{ansatz-refined-new} and the new expressions of $\phi^0$, $\phi^{1/2}$ and $\phi^1$ into \eqref{evoeq_flat}, we are able to rewrite the one-phase Muskat problem as
		\begin{equation}
		\begin{aligned}
			&\partial_t  \zeta + \partial_x\left((1+\eps \zeta)(\partial_x \zeta + \tfrac{\sqrt{\mu}}{\mathrm{bo}}\partial_{xxx} \zeta)\right)  + \tfrac{\mu}{3}\partial_{xx}\left((1+\eps\zeta)^3 \partial_{xx}\zeta\right) \\[5pt]&=-\tfrac{1}{\eps}\partial_x\left(\int_{-1}^0\left((1+\eps\zeta)\partial_x \widetilde{\phi}^\mu_{\rm ref} - \eps(z+1)\partial_x\zeta \partial_z\widetilde{\phi}^\mu_{\rm ref}\right)dz\right)
		\end{aligned}
		\end{equation}
		where $\widetilde{\phi}^\mu_{\rm ref}$ satisfies \eqref{laplace_phimu_ref}.  Dropping the terms of order $O(\mu^{3/2})$, we obtain the following fourth-order refined thin film equation 
			\begin{equation}\label{O(mu)}
		\begin{aligned}
		\partial_t  \zeta + \partial_x\left((1+\eps \zeta)(\partial_x \zeta + \tfrac{\sqrt{\mu}}{\mathrm{bo}}\partial_{xxx} \zeta)\right)  + \tfrac{\mu}{3}\partial_{xx}\left((1+\eps\zeta)^3 \partial_{xx}\zeta\right)=0. 
		\end{aligned}
		\end{equation}Notice that in the new regime the approximation of order $O(\mu^{3/2})$ does not capture the critical sixth-order term of \eqref{O(mu)_ill}. However, there is an additional term that refines the standard thin film equation \eqref{thin-film}. Let us now investigate the existence of the solution to \eqref{O(mu)}. First we introduce the definition of its weak solution.

\begin{definition}\label{weakdefO(mu)}
	We say that $\zeta\in L^\infty([0,T],L^\infty(\mathbb{T}))\cap L^1([0,T],W^{3,\infty}(\mathbb{T})) $ is a weak solution to \eqref{O(mu)} with initial data $\zeta_0\in L^\infty(\mathbb{T})$ if for any $\varphi\in C^\infty([0,T]\times \mathbb{T})$ and $t\in[0,T]$ the following equality holds:
	\begin{equation}
	\begin{aligned}
	&\int_\mathbb{T} \zeta(t,x)\varphi(t,x) dx-	\int_\mathbb{T} \zeta_0(x)\varphi(0,x)dx - \int_0^t\int_\mathbb{T} \zeta(t')\partial_t \varphi(t',x)dx dt'\\
	&\quad=\int_0^t\int_\mathbb{T}(1+\eps \zeta(t,x)) \partial_x  \zeta(t',x)\partial_x\varphi(t',x)dxdt'\\&\quad\quad +\frac{\sqrt{\mu}}{\rm bo}\int_0^t\int_\mathbb{T}(1+\eps \zeta(t,x)) \partial_{xxx} \zeta(t',x)\partial_x\varphi(t',x)dxdt' \\&\quad\quad-\frac{\mu}{3}\int_0^t\int_\mathbb{T}(1+\eps \zeta(t,x))^3 \partial_{xx} \zeta(t',x)\partial_{xx}\varphi(t',x)dxdt'
	\end{aligned}
	\end{equation}
\end{definition}

We are now able to show the existence of the solution to \eqref{O(mu)} in Wiener spaces.

\begin{theorem}\label{theoO(mu)}
	Let $\mathrm{bo}>0$ small enough and $0<\mu<1$ such that $\tfrac{\sqrt{\mu}}{\mathrm{bo}} +\tfrac{\mu}{3}>1$ and $\zeta_0\in \wiener[0]_0(\mathbb{T})$  a zero-mean function such that 
	\begin{equation}\label{smallnessTF_ref}
	|\zeta_0|_{0,0}< \tfrac{\tfrac{\sqrt{\mu}}{\mathrm{bo}}+\tfrac{\mu}{3}-1}{128\eps \left(\tfrac{\sqrt{\mu}}{\mathrm{bo}}+\tfrac{55}{6}\mu+1\right)}.\end{equation} Then there exists a global weak solution to \eqref{O(mu)} in the sense of Definition \ref{weakdefO(mu)} with initial data $\zeta_0$ which becomes instantaneously analytic in a growing strip in the complex plane. In particular, for $$\nu\in \left[0,\tfrac{1}{2}(\tfrac{\sqrt{\mu}}{\mathrm{bo}}+\tfrac{\mu}{3}-1)\right),$$ we have  $$\zeta\in L^\infty([0,T];\wiener[0]_{\nu t})\cap L^1([0,T];\wiener[4]_{\nu t})$$
	for all $T>0$ and the solution $\zeta$ satisfies for any $t\in[0,T]$ the energy inequality
	\begin{equation}\label{energy_estimate_ref}
	|\zeta(t)|_{0,\nu t} + \tfrac{1}{64} \left(\tfrac{\sqrt{\mu}}{\mathrm{bo}}+\tfrac{\mu}{3}-1\right)\int_0^t|\zeta(t')|_{4,\nu t'}dt'\leq	|\zeta_0|_{0,0}
	\end{equation} together with the decay
	\begin{equation}\label{decay_ref}
	|\zeta(t)|_{0,\nu t} \leq 	|\zeta_0|_{0,0} \ e^{-\frac{t}{64}\left(\tfrac{\sqrt{\mu}}{\mathrm{bo}}+\tfrac{\mu}{3}-1\right)}.
	\end{equation}
\end{theorem}
\begin{proof}
	Analogously to the proof of Theorem \ref{theoTF}, we derive a priori energy estimate and the existence is obtained by using a standard approximation argument that we omit. Considering the Fourier series of \eqref{O(mu)}, we get 
	\begin{equation}
		\partial_t \widehat{\zeta} +\left(\tfrac{\sqrt{\mu}}{\mathrm{bo}} +\tfrac{\mu}{3}\right)|n|^4 \widehat{\zeta} - |n|^2 \widehat{\zeta} = \widehat{\rm NL_2(\zeta)} + \widehat{\rm NL_3(\zeta)} +\widehat{\rm NL_4(\zeta)} ,
	\end{equation}where $\widehat{\zeta}$ denotes the Fourier coefficient of $\zeta$ and $\rm NL_2(\zeta)$, $\rm NL_3(\zeta)$, $\rm NL_4(\zeta)$ denote respectively the quadratic, the cubic and quartic nonlinear terms, namely
	\begin{equation*}\begin{aligned}
		{\rm NL_2(\zeta)}=  -\eps\bigg[&\left(\tfrac{\sqrt{\mu}}{\mathrm{bo}} + 2\mu\right) \partial_x\zeta \partial_{xxx}\zeta  +\left(\tfrac{\sqrt{\mu}}{\mathrm{bo}} + \mu\right) \zeta \partial_{xxxx}\zeta\\[5pt]&+\mu (\partial_{xx}\zeta)^2 +(\partial_x\zeta)^2 + \zeta\partial_{xx}\zeta\bigg], 
		\end{aligned}
	\end{equation*}
	\begin{equation*}\begin{aligned}
{	\rm NL_3(\zeta)}= -\mu\eps^2\left(2(\partial_x\zeta)^2 \partial_{xx}\zeta + 2\zeta(\partial_{xx}\zeta)^2 +4\zeta\partial_x\zeta\partial_{xxx}\zeta+ \zeta^2\partial_{xxxx}\zeta\right)
	\end{aligned}
	\end{equation*}
	and 	\begin{equation*}\begin{aligned}
{	\rm NL_4(\zeta)}=  -\mu \eps^3 \left( 2\zeta(\partial_x\zeta)^2 \partial_{xx}\zeta + \zeta^2(\partial_{xx}\zeta)^2 +2\zeta^2\partial_x\zeta\partial_{xxx}\zeta +\tfrac{1}{3}\zeta^3\partial_{xxxx}\zeta \right) 
	\end{aligned}
	\end{equation*}
	Choosing $\nu\in \left[0,\tfrac{1}{2}(\tfrac{\sqrt{\mu}}{\mathrm{bo}}+\tfrac{\mu}{3}-1)\right)$, it yields
	\begin{equation}
	\frac{d}{dt} | \zeta|_{0,\nu t} +\tfrac{1}{32}(\tfrac{\sqrt{\mu}}{\mathrm{bo}}+\tfrac{\mu}{3}-1)|\zeta|_{4,\nu t}\leq |\rm NL_2(\zeta)|_{0,\nu t}+ |\rm NL_3(\zeta)|_{0,\nu t}+ |\rm NL_4(\zeta)|_{0,\nu t}.\\[5pt]
	\label{apriori_ref}
	\end{equation} 
	Using Proposition \ref{productWie} for $s=0$, $\lambda=\nu t$ and  Proposition \ref{interpolation}, we have
	\begin{equation}\label{NL2_est}
	|\rm {NL_2(\zeta)}|_{0,\nu t} \leq 2\left(\tfrac{\sqrt{\mu}}{\rm bo} +2\mu +1\right)\eps|\zeta|_{0,\nu t}|\zeta|_{4,\nu t}, 
	\end{equation} 
		\begin{equation}\label{NL3_est}
	|\rm {NL_3(\zeta)}|_{0,\nu t} \leq 9\mu\eps^2|\zeta|^2_{0,\nu t}|\zeta|_{4,\nu t},
	\end{equation} 
	and 
		\begin{equation}\label{NL4_est}
	|\rm {NL_4(\zeta)}|_{0,\nu t} \leq \tfrac{16}{3}\mu\eps^3|\zeta|^3_{0,\nu t}|\zeta|_{4,\nu t}.
	\end{equation} 
	Injecting \eqref{NL2_est}-\eqref{NL4_est} into \eqref{apriori_ref}, we obtain
		\begin{equation*}
		\begin{aligned}
	\frac{d}{dt} | \zeta|_{0,\nu t} &+\tfrac{1}{32}(\tfrac{\sqrt{\mu}}{\mathrm{bo}}+\tfrac{\mu}{3}-1)|\zeta|_{4,\nu t}\\[5pt]&\leq \eps\left(2\left(\tfrac{\sqrt{\mu}}{\rm bo} +2\mu +1\right) + 9\mu\eps|\zeta|_{0,\nu t} + \tfrac{16}{3}\mu\eps^2|\zeta|^2_{0,\nu t} \right) |\zeta|_{0,\nu t}|\zeta|_{4,\nu t}
	\end{aligned}
	\end{equation*} 
	and, if we consider $\eps|\zeta|_{0,\nu t} < 1$, it implies 
		\begin{equation*}
	\begin{aligned}
	\frac{d}{dt} | \zeta|_{0,\nu t} &+\tfrac{1}{32}(\tfrac{\sqrt{\mu}}{\mathrm{bo}}+\tfrac{\mu}{3}-1)|\zeta|_{4,\nu t}\leq 2\eps\left(\tfrac{\sqrt{\mu}}{\rm bo} +\tfrac{55}{6}\mu +1\right)  |\zeta|_{0,\nu t}|\zeta|_{4,\nu t}.
	\end{aligned}
	\end{equation*} 
	Therefore, assuming $|\zeta|_{0,\nu t} < \tfrac{\tfrac{\sqrt{\mu}}{\mathrm{bo}}+\tfrac{\mu}{3}-1}{128\eps \left(\tfrac{\sqrt{\mu}}{\mathrm{bo}}+\tfrac{55}{6}\mu+1\right)},$ we are able to close the estimate and obtain 
		\begin{equation}\label{enest_ref}
	\begin{aligned}
	\frac{d}{dt} | \zeta|_{0,\nu t} &+\tfrac{1}{64}(\tfrac{\sqrt{\mu}}{\mathrm{bo}}+\tfrac{\mu}{3}-1)|\zeta|_{4,\nu t}\leq 0
	\end{aligned}
	\end{equation} 
	Following the argument used in the proof of Theorem \ref{theoTF}, one has that \eqref{enest_ref} is global in time and implies \eqref{energy_estimate_ref} and \eqref{decay_ref}.
\end{proof}

 \section{Elliptic estimate for the remainder potential}\label{ellipest_sec}
 In this section we focus on the remainder potential $\widetilde{\phi}^\mu$ in the ansatz \eqref{ansatz}-\eqref{ansatz-refined} we have previously introduced in \eqref{evoeq_flat}-\eqref{ellipro_flat} to obtain the thin film approximations \eqref{thin-film}-\eqref{O(mu)}. Our aim is to derive an elliptic estimate that we will use in the next section when we will rigorously justify the thin-film approximations derived in this paper.\\
 On one hand, using \eqref{ansatz}  the remainder potential satisfies the elliptic problem \eqref{laplace_phimu}. On the other hand, using \eqref{ansatz-refined} the remainder potential satisfies the elliptic problem \eqref{laplace_phimu_ref}. One can notice that both elliptic problems have the same structure 
 \begin{equation}\label{laplace_phimu2}
 \begin{cases}
 \nabla^\mu \cdot P(\Sigma)\nabla^\mu \widetilde{\phi}^\mu=f&\mbox{in}\quad \mathcal{S},\\[5pt]
 \widetilde{\phi}^\mu= h  & \mbox{on}\quad \mathbb{T}\times\{0\},\\[5pt]
\partial_z \widetilde{\phi}^\mu=0&\mbox{on}\quad \mathbb{T}\times\{-1\}.
 \end{cases}
 \end{equation}for different source terms $f$. In order to derive an elliptic estimate for $\widetilde{\phi}^\mu$, let us introduce the following proposition stating the existence in Wiener-Sobolev spaces of the solution to a Poisson problem with a source term that is a sum of a divergence form term and a function.
 
 \begin{proposition}\label{ellipestimate}
 	Let  $g\in \mathcal{A}_\lambda^{s+1,1}$, $f\in \mathcal{A}_\lambda^{s,0}$ and $h\in \mathbb{A}^{s}_\lambda$ with $s\geq 0$ and $\lambda\geq 0.$ There exists a unique $\varphi$ solution to 
 	\begin{equation}\label{ellipest}
 	\begin{cases}
 	\Delta^\mu \varphi= \nabla^\mu \cdot g + f&\mbox{in}\quad \mathcal{S},\\
 	\varphi=h& \mbox{on}\quad \mathbb{T}\times\{0\},\\
 	\partial_z \varphi =0&\mbox{on}\quad \mathbb{T}\times\{-1\}.
 	\end{cases}
 	\end{equation} 
 	Moreover, such solution satisfies the estimate 
 	\begin{equation}\label{ellip_estimate}
 		\|\nabla^\mu \varphi\|_{\mathcal{A}_\lambda^{s,0}}\leq C \left( \|g\|_{\mathcal{A}_\lambda^{s,0}} + \|f\|_{\mathcal{A}_\lambda^{s,0}} + |h|_{s,\lambda}\right)
 	\end{equation}with $C=10.$
 \end{proposition}
 \begin{proof}
 	The proof of this proposition is an adaptation of the proof of Theorem 3.1 of \cite{GG-BS2020}. Considering the Fourier series of the elliptic equation of \eqref{ellipest}, we obtain the sequence of ODEs
 	\begin{equation}
 	-\mu k^2 \widehat{\varphi}(k,z) + \partial_z^2\widehat{\varphi}(k,z) = i \sqrt{\mu} k \ \widehat{g_1}(k,z) + \partial_z \widehat{g_2}(k,z) + \widehat{f}(k,z), \qquad k\in \mathbb{Z}.
 	\end{equation}
 	Using the Dirichlet and Neumann boundary conditions of the elliptic problem, we can obtain the explicit solution to \eqref{ellipest}
 	
 	\begin{equation}
 	\begin{aligned}
 	\widehat{\varphi}(k,z)=\  &\tfrac{1}{\sqrt{\mu}|k|}\int_{-1}^{z} \Pi_1(\mu|k|, r, z) \left(i\sqrt{\mu}k \hat{g_1}(k,r) + \partial_r \hat{g_2}(k,r) + \hat{f}(k,r) \right)dr \\&+  \tfrac{1}{\sqrt{\mu}|k|}\int_z^0 \Pi_2 (\sqrt{\mu}|k|, r,z) \left(i\sqrt{\mu}k \hat{g_1}(k,r) + \partial_r \hat{g_2}(k,r) + \hat{f}(k,r) \right)ds\\& + \widehat{h}(k)\frac{e^{\sqrt{\mu}|k|(1+z)}+ e^{-\sqrt{\mu}|k|(1+z)}}{e^{\sqrt{\mu}|k|}+ e^{-\sqrt{\mu}|k|}}
 	\end{aligned}
 	\end{equation} with 
 	\begin{align*}
 	&	\Pi_1(\sqrt{\mu}|k|,s,z)=\frac{e^{\sqrt{\mu}|k|(1+s)}+ e^{-\sqrt{\mu}|k|(1+s)}}{e^{\sqrt{\mu}|k|}+ e^{-\sqrt{\mu}|k|}} \left(e^{\sqrt{\mu}|k|z}- e^{-\sqrt{\mu}|k|z}\right)\\[5pt]
 	&	\Pi_2(\sqrt{\mu}|k|,s,z)=e^{\sqrt{\mu}|k|(s-z)}- e^{-\sqrt{\mu}|k|(s-z)}
 	\end{align*}
 	Notice that $\Pi_{1,2}$ are the same functions as in \cite{GG-BS2020}. On the one hand, we can use the following estimates on $\Pi_{1}$ and $\Pi_{2}$  therein. For $(j,l)\in \mathbb{N}^2$ it holds
 	\begin{equation}\label{Pi1_bound}
 	\begin{aligned}
 	\int_{0}^{1} \|\mathbf{1}_{[-1,z]}(\cdot) \partial_z^j\partial_r^l \Pi_{1}(\sqrt{\mu}|k|,\cdot,z)\|_{L^\infty_r}dz \leq 2(\sqrt{\mu}|k|)^{j+l-1} 
 	\end{aligned}
 	\end{equation}
 		\begin{equation}\label{Pi2_bound}
 	\begin{aligned}
 	\int_{0}^{1} \|\mathbf{1}_{[z,0]}(\cdot) \partial_z^j\partial_r^l \Pi_{2}(\sqrt{\mu}|k|,\cdot,z)\|_{L^\infty_r}dz \leq \tfrac{5}{2}(\sqrt{\mu}|k|)^{j+l-1}
 	\end{aligned}
 	\end{equation} to obtain the bound for $\|\nabla^\mu \varphi\|_{\mathcal{A}^{s,0}_\lambda}$ with $\|g\|_{\mathcal{A}^{s,0}_\lambda}$ following the same argument as in \cite{GG-BS2020}. On the other hand, using the definition of $\Pi_{1,2}$ we know that $|\Pi_1(\sqrt{\mu}|k|,r,z)|\leq 1$ for $-1\leq r \leq z$  and $|\Pi_2(\sqrt{\mu}|k|,r,z)|\leq C$ for $z\leq r\leq 0$ and some constant $C>0.$ Using these two inequalities we get the bound for $\|\nabla^\mu \varphi\|_{\mathcal{A}^{s,0}_\lambda}$  with $\|f\|_{\mathcal{A}^{s,0}_\lambda}$.
  Let us detail only the bound coming from the non-homogeneous Dirichlet boundary condition. We have
 	
 	\begin{equation}
 	\begin{aligned}
 	\sqrt{\mu}|k| 	\int_{-1}^0 |\widehat{h}(k)|\frac{e^{\sqrt{\mu}|k|(1+z)}+ e^{-\sqrt{\mu}|k|(1+z)}}{e^{\sqrt{\mu}|k|}+ e^{-\sqrt{\mu}|k|}} &= \sqrt{\mu}|k| 	\int_{-1}^0 |\widehat{h}(k)| \frac{\cosh((1+z)\sqrt{\mu}|k|)}{\cosh(\sqrt{\mu}|k|)}dz \\[5pt]&
 	=|\hat{h}(k)|\tanh(\sqrt{\mu}|k|)\leq |\hat{h}(k)|
 	\end{aligned}
 	\end{equation}
 	and 
 	\begin{equation}
 	\begin{aligned}
 	&\int_{-1}^0 |\widehat{h}(k)|\frac{\left|\partial_z\left(e^{\sqrt{\mu}|k|(1+z)}+ e^{-\sqrt{\mu}|k|(1+z)}\right)\right|}{e^{\sqrt{\mu}|k|}+ e^{-\sqrt{\mu}|k|}} \\[5pt]&= \sqrt{\mu}|k| 	\int_{-1}^0 |\widehat{h}(k)| \frac{\sinh((1+z)\sqrt{\mu}|k|)}{\cosh(\sqrt{\mu}|k|)}dz 
 	=|\hat{h}(k)| \left(1- \frac{1}{\cosh(\sqrt{\mu}|k|)}\right)\leq |\hat{h}(k)|.
 	\end{aligned}
 	\end{equation}
 	These two inequalities give the bound for $\|\nabla^\mu \varphi\|_{\mathcal{A}^{s,0}_\lambda}$ with $|h|_{s,\lambda}$.  By linearity of the elliptic equation in \eqref{ellipest}, the estimate \eqref{ellip_estimate} follows putting together the three bounds obtained.
 \end{proof}
\begin{remark}Let us point out that the elliptic estimate \eqref{ellip_estimate} is not sharp. Indeed using \eqref{Pi1_bound}-\eqref{Pi2_bound} also for the bound of the part of $\widehat{\phi}$ containing $\widehat{f}$ together with the inequality 
$$\frac{(1 + |k|)^s}{\sqrt{\mu}|k|} \leq \frac{2}{\sqrt{\mu}}(1 +|k|)^{s-1} \qquad \mbox{for} \quad|k|\geq 1,$$
one would obtain $\tfrac{1}{\sqrt{\mu}} \|f\|_{\mathcal{A}_\lambda^{s-1,0}} $ instead of $ \|f\|_{\mathcal{A}_\lambda^{s,0}} $ in \eqref{ellip_estimate}. However, since we will use this elliptic estimate to prove a quantitative convergence result in Section \eqref{conv_sec}, the optimal order of convergence will be obtained  considering the estimate without the singularity $1/\sqrt{\mu}$. Therefore, the price one has to pay for the optimal order is that the source term of the corresponding elliptic system needs to be more regular.
\end{remark}
 In order to apply Proposition \ref{ellipestimate} to the elliptic problem \eqref{laplace_phimu2}, using the fact that $P(\Sigma)=\mathrm{Id} + Q(\Sigma)$, we write it as the following Poisson problem 
 	\begin{equation}\label{poisson_phimu}
 	\begin{cases}
 	\Delta^\mu  \widetilde{\phi}^\mu=-\nabla^\mu \cdot Q(\Sigma)\nabla^\mu\widetilde{\phi}^\mu +f&\mbox{in}\quad \mathcal{S},\\[5pt]
 	\widetilde{\phi}^\mu=h & \mbox{on}\quad \mathbb{T}\times\{0\},\\[5pt]
 	\frac{1}{1+\eps\zeta}\partial_z  \widetilde{\phi}^\mu=0& \mbox{on}\quad \mathbb{T}\times\{-1\},\\
 	\end{cases}
 	\end{equation}
 	Following the approach in \cite{GG-BS2020}, we aim to derive an estimate on the $\mathcal{A}^{s,0}_\lambda$-norm of $\nabla^\mu \widetilde{\phi}^\mu$ only in terms of Wiener norms of the free surface $\zeta$. To do that, we need to introduce the next two proposition that give regularity estimates in terms of $\zeta$ of the matrix $Q(\Sigma)$ and the function $F_{\eps\sqrt{\mu}}(\partial_x\zeta)$ appearing respectively in the divergence term and in the Dirichlet boundary condition of \eqref{poisson_phimu}.

 \begin{proposition}\label{prop_Q}
 	Let $Q(\Sigma)$ be as in \eqref{defQ}. Assume that $\mu < 1 $ and  $|\zeta|_{1,\lambda}\leq \frac{1}{8K_s\eps}$ with $s,\lambda\geq 0$. Then, choosing $\sigma(x,z)=\eps\zeta(x)(z+1)$,  we have 
 	\begin{equation}\label{estimateQ}
 	\|Q(\Sigma)\|_{\mathcal{A}_\lambda^{s,1}}\leq 3\eps \sqrt{\mu}|\zeta|_{s+1,\lambda}.
 	\end{equation}
 \end{proposition}
 \begin{proof}
 	With the particular choice $ \sigma(x,z)=\eps\zeta(x)(z+1)$, $Q(\Sigma)$ reads 
 	$$Q(\Sigma)=\begin{pmatrix}
 	\eps\zeta& -\eps\sqrt{\mu} \partial_x\zeta (z+1)\\[10pt]
 	-\eps\sqrt{\mu} \partial_x\zeta (z+1)&\dfrac{-\eps\zeta + \eps^2\mu(\partial_x \zeta)^2(z+1)^2}{1+\eps\zeta}
 	\end{pmatrix}.$$
 	Following \cite{GG-BS2020}, we use the expansion of $\frac{1}{1+x}$ for $|x|<1$ to write $Q(\Sigma)$ as the infinite sum 
 	$$Q(\Sigma)=\eps \sum\limits_{n\geq 0}\eps^n Q_{(n)}(\zeta)$$
 	with 
 	$$Q_{(0)}(\zeta)=\begin{pmatrix}
 	\zeta& -\sqrt{\mu} \partial_x\zeta (z+1)\\[10pt]
 	-\eps\sqrt{\mu} \partial_x\zeta (z+1)&-\zeta + \eps\mu(\partial_x \zeta)^2(z+1)^2
 	\end{pmatrix}$$and $$ Q_{(n)}(\zeta)=\begin{pmatrix}
 	\quad0& 0\\[10pt]
 	\quad 0&(-\zeta)^n\left(-\zeta + \eps\mu(\partial_x \zeta)^2(z+1)^2\right)
 	\end{pmatrix}.$$
 	We have 
 	\begin{equation}
 	\begin{aligned}
 	\|Q_{(0)}(\zeta)\|_{A^{s,1}_\lambda} &\leq 2\sqrt{\mu}\|\partial_x \zeta (z+1)\|_{A^{s,1}_\lambda} + \eps\mu \|(\partial_x\zeta)^2(z+1)^2\|_{A^{s,1}_\lambda} \\
 	& \leq 2\sqrt{\mu}|\zeta|_{s+1,\lambda} + 2\eps\mu|(\partial_x\zeta)^2|_{s,\lambda}\\
 	&\leq 2\sqrt{\mu}|\zeta|_{s+1,\lambda} + 2\eps\mu K_{s,2}|\zeta|_{1,\lambda}|\zeta|_{s+1,\lambda}\\& = 2\left(\sqrt{\mu}
 	+ \eps\mu K_{s}|\zeta|_{1,\lambda}\right)|\zeta|_{s+1,\lambda}
 	\end{aligned}
 	\end{equation}
 	and 
 	\begin{equation}
 	\begin{aligned}
 	\|Q_{(n)}(\zeta)\|_{A^{s,1}_\lambda} &\leq \eps \mu\|\zeta^n(\partial_x \zeta)^2 (z+1)^2\|_{A^{s,1}_\lambda} \leq 2 \eps \mu |\zeta^n (\partial_x \zeta)^2|_{s,\lambda}\\
 	&\leq 2\eps \mu K_s \left(|\zeta^n|_{0,\lambda}|(\partial_x\zeta)^2|_{s,\lambda} + |\zeta^n|_{s,\lambda}|(\partial_x \zeta)^2|_{0,\lambda}\right)\\
 	& \leq 2\eps \mu K_s\bigg(K_{0,n}|\zeta|^n_{0,\lambda} K_{s,2}|\partial_x \zeta|_{0,\lambda}|\partial_x\zeta|_{s,\lambda} \\ &+ K_{s,n}|\zeta|^{n-1}_{0,\lambda}|\zeta|_{s,\lambda}K_{0,2}|\partial_x\zeta|^2_{0,\lambda}\bigg)\\
 	&\leq 2\eps \mu K_s (2nK_s + 2K_{s,n}) |\zeta|^{n+1}_{1,\lambda}|\zeta|_{s+1,\lambda}.
 	\end{aligned}
 	\end{equation}
 	From the definition of $K_{s,n}$ one has 
 	$2K_s+K_{s,n}\leq 2 K_{s,n}$ and using the fact that $n\leq 2^n$, $2\leq 2^n$ for any $n\geq 1 $ we obtain 
 	\begin{equation}
 	\begin{aligned}
 	\|Q_{(n)}(\zeta)\|_{A^{s,1}_\lambda} &\leq2^{n+1}K_{s,n} |\zeta|^{n}_{1,\lambda}      2\eps \mu  K_s |\zeta|_{1,\lambda}|\zeta|_{s+1,\lambda}.
 	\end{aligned}
 	\end{equation}
 	Since we have 
 	\begin{equation}
 	\begin{aligned}
 	2^{n+1}K_{s,n}|\zeta|^n_{1,\lambda}= 2(2K_{s,n}^{1/n}|\zeta|_{1,\lambda})^n \leq 2 (4K_s |\zeta|_{1,\lambda})^n ,
 	\end{aligned}
 	\end{equation}if we assume that  $|\zeta|_{1,\lambda}\leq \frac{1}{8K_s\eps}$ we get 
 	$$	\|Q_{(n)}(\zeta)\|_{A^{s,1}_\lambda} \leq \frac{2}{2^n \eps^n}2\eps \mu K_s|\zeta|_{1,\lambda}  |\zeta|_{s+1,\lambda}.$$ Hence the series $\eps \sum\limits_{n\geq 0} \eps^n Q_{(n)}(\zeta)$ is summable and 
 	\begin{equation}
 	\begin{aligned}
 	\|Q(\Sigma)\|_{A^{s,1}_\lambda}&\leq  	\eps\|Q_{(0)}(\zeta)\|_{A^{s,1}_\lambda} + \eps \sum\limits_{n\geq 1}	\eps^{n}\|Q_{(n)}(\zeta)\|_{A^{s,1}_\lambda} \\&
 	\leq 2\eps\left(\sqrt{\mu}
 	+ \eps\mu K_{s}|\zeta|_{1,\lambda}\right)|\zeta|_{s+1,\lambda} + 2\eps  \sum\limits_{n\geq 1} \frac{1}{2^n}2\eps \mu K_s|\zeta|_{1,\lambda}  |\zeta|_{s+1,\lambda}\\
 	&\leq 2\eps [\sqrt{\mu} + \eps\mu K_s|\zeta|_{1,\lambda}   (1+2 \sum\limits_{n\geq 1}2^{-n} ) ]|\zeta|_{s+1,\lambda}\\&
 	\leq 2 \eps [\sqrt{\mu} +3 \eps\mu K_s|\zeta|_{1,\lambda}   ]|\zeta|_{s+1,\lambda}.
 	\end{aligned}
 	\end{equation}
 	Assuming $|\zeta|_{1,\lambda}\leq \frac{1}{8K_s\eps}$, one has $$3 \eps\mu K_s|\zeta|_{1,\lambda}\leq \frac{3}{8}\mu $$ and we have 
 	\begin{equation*}
 	\|Q(\Sigma)\|_{A^{s,1}_\lambda}\leq 2\eps (\sqrt{\mu} +\frac{3}{8}\mu )|\zeta|_{s+1,\lambda} \leq \frac{11}{4}\eps \sqrt{\mu} |\zeta|_{s+1,\lambda}\leq 3\eps \sqrt{\mu} |\zeta|_{s+1,\lambda}.
 	\end{equation*}
 	where we have used in the second inequality the fact that $\mu < 1.$
 \end{proof}
 
 The control of the Wiener norm of $F_{\eps\sqrt{\mu}}(\partial_x\zeta)$ in terms of the Wiener norm of $\zeta$ can be derived by exploiting the analiticity of $F_{\eps\sqrt{\mu}}$ and its explicit expression.

\begin{proposition}\label{prop_Fepsmu} 
 	Let $F_{\eps\sqrt{\mu}}(\cdot)$ as in \eqref{defcurv_dimless} and $\mu < 1$. If $\zeta\in\mathbb{A}^{s+1}_\lambda$ with $s,\lambda\geq 0$ such that $36K^2_s \eps|\zeta|_{1,\lambda}\leq 1$, then 
 	\begin{equation}\label{est_Fepsmu}
 	|F_{\eps\sqrt{\mu}}(\partial_x \zeta)|_{s,\lambda}\leq  \eps\mu |\zeta|_{s+1,\lambda}.
 	\end{equation}
 \end{proposition}
 \begin{proof}
 	From the definition of $F_{\eps\sqrt{\mu}}(\partial_x\zeta)$ it yields
 	$$F_{\eps\sqrt{\mu}}(\partial_x\zeta)= F(\eps\sqrt{\mu}\partial_x\zeta)= G(\eps^2\mu (\partial_x\zeta)^2)$$ with $G(\cdot)$ as in Lemma \ref{estimate_G}. Therefore,  we get
 	\begin{equation*}
 	\begin{aligned}
 	|F_{\eps\sqrt{\mu}}(\partial_x \zeta)|_{s,\lambda}&\leq 18 K_s\eps^2\mu|(\partial_x \zeta)^2|_{s,\lambda}
 	\end{aligned}
 	\end{equation*}
 	From Proposition \ref{powerWie} we have 
 	\begin{equation*}
 	\begin{aligned}
 	|F_{\eps\sqrt{\mu}}(\partial_x \zeta)|_{s,\lambda}&\leq 18 K_s \eps^2\mu K_{s,2}|\partial_x \zeta|_{0,\lambda} |\partial_x \zeta|_{s,\lambda} 
 	\leq  36 \eps^2\mu K^2_s |\zeta|_{1,\lambda}|\zeta|_{s+1,\lambda}
 	\end{aligned}
 	\end{equation*}and \eqref{est_Fepsmu} follows using the smallness assumption on $|\zeta|_{1,\lambda}$.
 \end{proof}

 We are now able to derive the regularity estimate on $\nabla^\mu \widetilde{\phi}^\mu$ in $\mathcal{A}^{s,0}_\lambda$ for both elliptic problems \eqref{laplace_phimu} and \eqref{laplace_phimu_ref}.
 
 \begin{proposition}\label{ell_est}
 	Let $\widetilde{\phi}^\mu$ be the solution to the elliptic problem \eqref{laplace_phimu} with $\mu < 1$ and $\sigma(x,z)=\eps\zeta(x)(z+1)$.  Then, for $\lambda\geq 0$, $s>0$ and provided 
 \begin{equation}\label{assump-zeta1}
 		\max(6K_s, C) 6K_s\eps |\zeta|_{1,\lambda} \leq 1
 		\end{equation}
 	with $C$ as in Proposition \ref{ellipestimate}, the following estimate holds:
 	\begin{equation}\label{ellipticestimate_s}
 	\begin{aligned}
 	\|\nabla^\mu \widetilde{\phi}^\mu\|_{\mathcal{A}_\lambda^{s,0}} \leq 8C\eps\mu \left(|\zeta|_{s+2,\lambda} +\tfrac{1}{\rm Bo} |\zeta|_{s+4,\lambda}\right).
 	\end{aligned}
 	\end{equation}
 	Moreover, if $6C\eps |\zeta|_{1,\lambda}\leq 1$ one has
 	\begin{equation}\label{ellipticestimate_0}
 	\begin{aligned}
 	\|\nabla^\mu \widetilde{\phi}^\mu\|_{\mathcal{A}_\lambda^{0,0}} \leq 4C \eps\mu  \left(|\zeta|_{2,\lambda} +\tfrac{1}{\rm Bo} |\zeta|_{4,\lambda}\right).
 	\end{aligned}
 	\end{equation}
 \end{proposition}
 
 \begin{proof}	
 	Applying Proposition \ref{ellipestimate} to the Poisson problem \eqref{poisson_phimu} with 
 	$$f= -\mu A_\zeta(\partial_x,\partial_z)\phi^0, \qquad h=\tfrac{\eps}{\rm Bo}F_{\eps\sqrt{\mu}}(\partial_x \zeta)\partial_{xx}\zeta$$ as in \eqref{laplace_phimu} (recall that $\phi^{1/2}\equiv 0$), we obtain the following estimate 
 	\begin{equation}
 	\begin{aligned}
 	\|\nabla^\mu \widetilde{\phi}^\mu\|_{\mathcal{A}_\lambda^{s,0}}
 	\leq C\big( & \|Q( \Sigma)\nabla^\mu \widetilde{\phi}^\mu\|_{\mathcal{A}^{s,0}_\lambda}\\& + \mu \|A_\zeta(\partial_x,\partial_z)\phi^0\|_{\mathcal{A}^{s,0}_\lambda}   + \tfrac{\eps}{\rm Bo} | F_{\eps\sqrt{\mu}}(\partial_x \zeta)\partial_{xx}\zeta|_{s,\lambda}  \big)
 	\end{aligned}
 	\end{equation}
 	From \eqref{Asigma} we have that
 	$$\|A_\zeta(\partial_x,\partial_z)\phi^0 \|_{\mathcal{A}^{s,0}_\lambda}=\eps|(1+\eps\zeta)( \partial_{xx}\zeta +\tfrac{1}{\rm Bo}\partial_{xxxx}\zeta) |_{s,\lambda}.$$
 	Using Proposition \ref{productWie} and Proposition \ref{productWieSob} we get
 	\begin{equation}\label{est_intermediate}
 	\begin{aligned}
 	&\|\nabla^\mu \widetilde{\phi}^\mu\|_{\mathcal{A}_\lambda^{s,0}}
 	\leq C\bigg(  K_s\|Q( \Sigma)\|_{\mathcal{A}^{0,1}}\|\nabla^\mu \widetilde{\phi}^\mu\|_{\mathcal{A}^{s,0}_\lambda} + K_s\|Q( \Sigma)\|_{\mathcal{A}^{s,1}}\|\nabla^\mu \widetilde{\phi}^\mu\|_{\mathcal{A}^{0,0}_\lambda}\\&+ \eps\mu | \partial_{xx}\zeta +\tfrac{1}{\rm Bo}\partial_{xxxx}\zeta |_{s,\lambda} +  \eps\mu | \eps\zeta(\partial_{xx}\zeta +\tfrac{1}{\rm Bo}\partial_{xxxx}\zeta) |_{s,\lambda}  \\& + K_s\tfrac{\eps}{\rm Bo} \big(| F_{\eps\sqrt{\mu}}(\partial_x \zeta)|_{s,\lambda} |\partial_{xx}\zeta|_{0 ,\lambda} +  | F_{\eps\sqrt{\mu}}(\partial_x \zeta)|_{0,\lambda} |\partial_{xx}\zeta|_{s ,\lambda}\big)  \bigg)
 	\end{aligned}
 	\end{equation}
 	Furthermore we have for $s=0$ and using \eqref{est_Fepsmu}
 	\begin{equation}
 	\begin{aligned}
 	\|\nabla^\mu \widetilde{\phi}^\mu\|_{\mathcal{A}_\lambda^{0,0}}\leq  C & \big( \|Q(\Sigma )\|_{\mathcal{A}_\lambda^{0,1}}\|\nabla^\mu \widetilde{\phi}^\mu \|_{\mathcal{A}_\lambda^{0,0}} + \eps\mu(|\partial_{xx}\zeta|_{0,\lambda} + \eps |\zeta \partial_{xx}\zeta|_{0,\lambda} )\\&+ \tfrac{\eps\mu}{\rm Bo}(|\partial_{xxxx}\zeta|_{0,\lambda} + \eps|\zeta \partial_{xxxx}\zeta|_{0,\lambda}) + \tfrac{\eps^2\mu}{\rm Bo} |\zeta|_{1,\lambda}|\partial_{xx}\zeta|_{0,\lambda}
 	\big).
 	\end{aligned}
 	\end{equation}
 	Using \eqref{estimateQ} we can move the term with $\widetilde{\varphi}^\mu$  to the left-hand side to obtain 
 	\begin{equation}\label{est_phi_wienersob0}
 	\begin{aligned}
 	&(1- 3C\eps\sqrt{\mu} |\zeta|_{1,\lambda})\|\nabla^\mu \widetilde{\phi}^\mu\|_{\mathcal{A}_\lambda^{0,0}} \\&\leq C\left( \eps\mu (1+\eps |\zeta|_{0,\lambda})|\zeta|_{2,\lambda} + \tfrac{\eps\mu}{\rm Bo} (1+\eps|\zeta|_{0,\lambda})|\zeta|_{4,\lambda} +\tfrac{\eps^2\mu}{\rm Bo} |\zeta|_{1,\lambda}|\zeta|_{2,\lambda}
 	\right)
 	\end{aligned}
 	\end{equation}
 	and from \eqref{assump-zeta1} we get 
 	\begin{equation}\label{est_gradphi_wienersob0}
 	\begin{aligned}
 	\|\nabla^\mu \widetilde{\phi}^\mu\|_{\mathcal{A}_\lambda^{0,0}} \leq \frac{C\eps\mu(1+2\eps|\zeta|_{0,\lambda})}{(1- 3C\eps\sqrt{\mu} |\zeta|_{1,\lambda})}  \left(|\zeta|_{2,\lambda}+\tfrac{1}{\rm Bo}|\zeta|_{4,\lambda}\right).
 	\end{aligned}
 	\end{equation}
 	For $s\neq0$, using \eqref{est_gradphi_wienersob0}, Proposition \ref{prop_Q}, Proposition \ref{prop_Fepsmu} and Proposition \ref{productWie}  we get 
 	\begin{equation*}
 	\begin{aligned}
 	(1-3&CK_s \eps \sqrt{\mu}|\zeta|_{1,\lambda})\|\nabla^\mu \widetilde{\phi}^\mu\|_{\mathcal{A}_\lambda^{s,0}} \\[5pt] \leq  C  \bigg[ &3K_s\eps\sqrt{\mu}|\zeta|_{s+1,\lambda} \frac{C\eps\mu (1+2\eps|\zeta|_{0,\lambda})}{(1- 3C\eps\sqrt{\mu} |\zeta|_{1,\lambda})}  \left(|\zeta|_{2,\lambda}+\tfrac{1}{\rm Bo}|\zeta|_{4,\lambda} \right)\\[5pt]&	
 	+	\eps\mu(|\zeta|_{s+2,\lambda} + \tfrac{1}{\rm Bo}|\zeta|_{s+4,\lambda}  )+ \eps^2\mu(|\zeta\partial_{xx}\zeta|_{s,\lambda}+ \tfrac{1}{\rm Bo}|\zeta\partial_{xxxx}\zeta|_{s,\lambda})
 	\\[5pt]& + K_s \tfrac{\eps^2\mu}{\rm Bo}(|\zeta|_{s+1,\lambda}|\zeta|_{2,\lambda} + |\zeta|_{1,\lambda}|\zeta|_{s+2,\lambda}) \bigg]
 	\end{aligned}
 	\end{equation*}
 From Proposition \ref{productWie} and Proposition \ref{interpolation} it follows
 	\begin{equation*}
 	\begin{aligned}
 	&(1-3CK_s \eps \sqrt{\mu}|\zeta|_{1,\lambda})\|\nabla^\mu \widetilde{\phi}^\mu\|_{\mathcal{A}_\lambda^{s,0}} \\[5pt] &\leq  C  \bigg[ 3K_s\eps\sqrt{\mu}|\zeta|_{s+1,\lambda} \tfrac{C\eps\mu (1+2\eps|\zeta|_{0,\lambda})}{(1- 3C\eps\sqrt{\mu} |\zeta|_{1,\lambda})}  \left(|\zeta|_{2,\lambda}+\tfrac{1}{\rm Bo}|\zeta|_{4,\lambda} \right)\\[5pt]&\qquad	
 	+\eps\mu(1+2\eps K_s|\zeta|_{0,\lambda})(|\zeta|_{s+2,\lambda} + \tfrac{1}{\rm Bo}|\zeta|_{s+4,\lambda}) 
 	+ 2K_s \tfrac{\eps^2\mu}{\rm Bo} |\zeta|_{1,\lambda}|\zeta|_{s+2,\lambda} \bigg]\\[5pt]& \leq  
 	C \eps\mu(1+4\eps K_s|\zeta|_{0,\lambda})\left(\tfrac{3K_sC\eps\sqrt{\mu}|\zeta|_{1,\lambda}}{1-3C\eps\sqrt{\mu}|\zeta|_{1,\lambda}} +1\right) \left(|\zeta|_{s+2,\lambda} +\tfrac{1}{\rm Bo} |\zeta|_{s+4,\lambda}\right)
 	\end{aligned}
 	\end{equation*}
 	and  using \eqref{assump-zeta1} 
 	\begin{equation*}
 	\begin{aligned}
 	(1-\tfrac{ \sqrt{\mu}}{2})\|\nabla^\mu \widetilde{\phi}^\mu\|_{\mathcal{A}_\lambda^{s,0}} \leq 2C \eps\mu \left(\tfrac{\tfrac{\sqrt{\mu}}{2}}{1-\tfrac{\sqrt\mu}{2}} +1\right) \left(|\zeta|_{s+1,\lambda} +\tfrac{\sqrt{\mu}}{\rm bo} |\zeta|_{s+3,\lambda}\right).
 	\end{aligned}
 	\end{equation*}
 	For $\mu < 1$, we have $1-\tfrac{\sqrt{\mu}}{2}>\tfrac{1}{2}$ and 
 	\begin{equation*}
 	\begin{aligned}
 	\|\nabla^\mu \widetilde{\phi}^\mu\|_{\mathcal{A}_\lambda^{s,0}} \leq 8C\eps\mu \left(|\zeta|_{s+2,\lambda} +\tfrac{1}{\rm Bo} |\zeta|_{s+4,\lambda}\right).
 	\end{aligned}
 	\end{equation*}
 	Finally, using again \eqref{assump-zeta1} we obtain from \eqref{est_gradphi_wienersob0} 
 	\begin{equation*}
 	\begin{aligned}
 	\|\nabla^\mu \widetilde{\phi}^\mu\|_{\mathcal{A}_\lambda^{0,0}} \leq 4C\eps \mu \left(|\zeta|_{2,\lambda} +\tfrac{1}{\rm Bo} |\zeta|_{4,\lambda}\right).
 	\end{aligned}
 	\end{equation*}
 \end{proof}

 \begin{proposition}\label{ell_est_ref}
 	Let $\widetilde{\phi}^\mu_{\rm ref}$ be the solution to the elliptic problem \eqref{laplace_phimu_ref} with $\mu < 1$ and $\sigma(x,z)=\eps\zeta(x)(z+1)$.  Then, for $\lambda\geq 0$, $s>0$ and provided
 	\begin{equation}\label{assump-zeta1ref}
\max(6K_s,C)	6K_s\eps|\zeta|_{1,\lambda}\leq1
 	\end{equation}
the following estimate holds:
 	\begin{equation}\label{ellipticestimate_s_ref}
 	\begin{aligned}
 	\|\nabla^\mu \widetilde{\phi}^\mu_{\rm ref}\|_{\mathcal{A}_\lambda^{s,0}} \leq  C_s\eps\mu^{3/2}(\tfrac{1}{\rm bo}+1) |\zeta|_{s+4,\lambda}.
 	\end{aligned}
 	\end{equation}for some constant $C_s>0$ independent of $\mu$.
 	Moreover, if $6C\eps|\zeta|_{1,\lambda}\leq 1$ one has
 	\begin{equation}\label{ellipticestimate_0_ref}
 	\begin{aligned}
 	\|\nabla^\mu \widetilde{\phi}^\mu_{\rm ref}\|_{\mathcal{A}_\lambda^{0,0}} \leq C_0\eps\mu^{3/2}(\tfrac{1}{\rm bo}+1) |\zeta|_{4,\lambda}
 	\end{aligned}
 	\end{equation}for some constant $C_0>0$ independent of $\mu$.
 \end{proposition}
 
 \begin{proof}	
 	Applying Proposition \ref{ellipestimate} to the Poisson problem \eqref{poisson_phimu} with 
 	$$f= -\mu^{3/2}A_\zeta(\partial_x,\partial_z)\phi^{1/2} -\mu^2A_\zeta(\partial_x,\partial_z)\phi^{1}, \qquad h=\tfrac{\eps\sqrt{\mu}}{\rm Bo}F_{\eps\sqrt{\mu}}(\partial_x\zeta)\partial_{xx}\zeta$$  as in \eqref{laplace_phimu_ref}, we obtain the following estimate 
 	\begin{equation}\label{first_est_graphimu_ref}
 	\begin{aligned}
 	\|\nabla^\mu \widetilde{\phi}^\mu_{\rm ref}\|_{\mathcal{A}_\lambda^{s,0}}
 	\leq C\big( & \|Q( \Sigma)\nabla^\mu \widetilde{\phi}^\mu_{\rm ref}\|_{\mathcal{A}^{s,0}_\lambda} + \mu^{3/2} \|A_\zeta(\partial_x,\partial_z)(\phi^{1/2}+\sqrt{\mu}\phi^1)\|_{\mathcal{A}^{s,0}_\lambda}\\& + \tfrac{\eps\sqrt{\mu}}{\rm bo} | F_{\eps\sqrt{\mu}}(\partial_x \zeta)\partial_{xx}\zeta|_{s,\lambda}  \big)
 	\end{aligned}
 	\end{equation} where we have used the fact that $A_\zeta(\partial_x,\partial_z)$ is linear in its argument. From \eqref{Asigma} and since $\mu<1$ it follows
 	\begin{equation}\label{est_Aphi12phi1}
 	\begin{aligned}
 	&\|A_\zeta(\partial_x,\partial_z)(\phi^{1/2} +  \sqrt{\mu}\phi^1)\|_{\mathcal{A}^{s,0}_\lambda}\\[5pt]
 	&\leq \tfrac{\eps}{\rm bo} |(1+\eps\zeta)\partial_{xxxx}\zeta|_{s,\lambda} +\|A_\zeta(\partial_x,\partial_z)\phi^{1}\|_{\mathcal{A}^{s,0}_\lambda}.
 	\end{aligned} 
 	\end{equation}
 	Using \eqref{phi1} and the definition of $A_\zeta(\partial_x,\partial_z)$ one can compute that 
 	\begin{equation}
 		\begin{aligned}
 		A_\zeta(\partial_x,\partial_z)\phi^1=& -\eps\left(\tfrac{z^2}{2}+z\right)(1+\eps\zeta)^3\partial_{xxxx}\zeta + 2 \eps^2(1+\eps\zeta)^2\partial_x \zeta \partial_{xxx}\zeta \\[5pt]&+ \eps^2(1+\eps\zeta)^2 (\partial_{xx}\zeta)^2+ \eps^3 (1+\eps\zeta)(\partial_x\zeta)^2\partial_{xx}\zeta.
 		\end{aligned}
 	\end{equation} 	Roughly speaking  $\|A_\zeta(\partial_x,\partial_z)\phi^{1}\|_{\mathcal{A}_\lambda^{s,0}}$ is the only term that changes with respect to \eqref{est_intermediate}. Hence we only focus on it while the remaining terms can be treated as in the proof of Proposition \ref{ell_est}. Then, by definition of the $\mathcal{A}^{s,0}_\lambda$-norm we have 
 	\begin{equation}\begin{aligned}
 		\|A_\zeta(\partial_x,\partial_z)\phi^{1}\|_{\mathcal{A}^{s,0}_\lambda}\leq \ & \tfrac{\eps}{3}|(1+\eps\zeta)^3\partial_{xxxx}\zeta|_{s,\lambda} +  2 \eps^2|(1+\eps\zeta)^2\partial_x \zeta \partial_{xxx}\zeta|_{s,\lambda}\\[5pt]& + \eps^2|(1+\eps\zeta)^2 (\partial_{xx}\zeta)^2|_{s,\lambda}+ \eps^3 |(1+\eps\zeta)(\partial_x\zeta)^2\partial_{xx}\zeta|_{s,\lambda}.
 	\end{aligned}
 	\end{equation}
 	 Using Proposition \ref{productWie}, Proposition \ref{powerWie} and Proposition \ref{interpolation} we get 
 	\begin{equation*}
 		\begin{aligned}
 		\|A_\zeta(\partial_x,\partial_z)\phi^{1}\|_{\mathcal{A}^{s,0}_\lambda}\leq \eps\bigg[ &\tfrac{1}{3} + (6K_{s} + K_{s,2})\eps|\zeta|_{0,\lambda} \\[5pt]&+ \left(K_{s,2}+1 + 6K_{s}(2K_{s}+1) + K_{s}(K_{s,2}+1)\right)\eps^2|\zeta|^2_{0,\lambda}\\[5pt]& + \big(\tfrac{1}{3}(K_{s,3}+1) +2K_{s}(K_{s,2}+2K_{s}) + 2K_{s}K_{s,2} \\[5pt]&+ K_{s}(1+K_{s}(K_{s,2}+1))\big)\eps^3|\zeta|^3_{0,\lambda}     \bigg]|\zeta|_{s+4,\lambda}.
 		\end{aligned}
 	\end{equation*}
 	From \eqref{assump-zeta1ref} we know that $\eps|\zeta|_{0,\lambda}\leq \eps|\zeta|_{1,\lambda}<1$ and recalling the definition of $K_{s,n}$ in \eqref{est_wie_power} it yields
\begin{equation*}
\begin{aligned}
\|A_\zeta(\partial_x,\partial_z)\phi^{1}\|_{\mathcal{A}^{s,0}_\lambda}&\leq \eps\left(1+ P(K_{s})\eps|\zeta|_{0,\lambda}  \right)|\zeta|_{s+4,\lambda}
\end{aligned}
\end{equation*}where $P(K_{s})$ is some polynomial function in $K_{s}$. 
Therefore starting from \eqref{first_est_graphimu_ref}, repeating the same computations as in the proof of Proposition \ref{ell_est} and using  \eqref{est_Aphi12phi1}, 
we derive for $s>0$
	\begin{equation*}
\begin{aligned}
&(1-3CK_s \eps \sqrt{\mu}|\zeta|_{1,\lambda})\|\nabla^\mu \widetilde{\phi}^\mu_{\rm ref}\|_{\mathcal{A}_\lambda^{s,0}}\\[5pt]& \leq  
C \eps\mu^{3/2}(\tfrac{1}{\rm bo}+1)(1+P(K_s) \eps|\zeta|_{1,\lambda})\left(\tfrac{3K_sC\eps\sqrt{\mu}|\zeta|_{1,\lambda}}{1-3C\eps\sqrt{\mu}|\zeta|_{1,\lambda}} +1\right) |\zeta|_{s+4,\lambda}.
\end{aligned}
\end{equation*}
The hypothesis \eqref{assump-zeta1ref} gives
\begin{equation*}
\begin{aligned}
(1-\tfrac{ \sqrt{\mu}}{2})\|\nabla^\mu \widetilde{\phi}^\mu\|_{\mathcal{A}_\lambda^{s,0}} \leq C (1+\tfrac{P(K_s)}{6CK_s}) \eps\mu^{3/2} (\tfrac{1}{\rm bo}+1) \left(1+\tfrac{\tfrac{\sqrt{\mu}}{2}}{1-\tfrac{\sqrt{\mu}}{2}} \right)  |\zeta|_{s+4,\lambda}.
\end{aligned}
\end{equation*} and since $\mu < 1$ we have $1-\tfrac{\sqrt{\mu}}{2}> \tfrac{1}{2}$ and 
 	\begin{equation*}
 	\begin{aligned}
 	\|\nabla^\mu \widetilde{\phi}^\mu\|_{\mathcal{A}_\lambda^{s,0}} \leq C_s\eps\mu^{3/2}(\tfrac{1}{\rm bo}+1) |\zeta|_{s+4,\lambda}
 	\end{aligned}
 	\end{equation*} with $C_s=4C (1+\tfrac{P(K_s)}{6CK_s})$.
 	Moreover, when $s=0$ we obtain 
 		\begin{equation}\label{est_gradphi_ref_wienersob0}
 	\begin{aligned}
 	\|\nabla^\mu \widetilde{\phi}^\mu_{\rm ref}\|_{\mathcal{A}_\lambda^{0,0}} \leq C\eps\mu^{3/2}(\tfrac{1}{\rm bo}+1) \frac{ 1+23\eps|\zeta|_{1,\lambda}}{(1- 3C\eps\sqrt{\mu} |\zeta|_{1,\lambda})} |\zeta|_{4,\lambda}
 	\end{aligned}
 	\end{equation}and \eqref{assump-zeta1ref}, using again that $1-\tfrac{\sqrt{\mu}}{2}> \tfrac{1}{2},$ implies
 	\begin{equation*}
 	\begin{aligned}
 	\|\nabla^\mu \widetilde{\phi}^\mu\|_{\mathcal{A}_\lambda^{0,0}} \leq C_0\eps \mu^{3/2} (\tfrac{1}{\rm bo}+1)  |\zeta|_{4,\lambda}
 	\end{aligned}
 	\end{equation*}with $C_0=2C(1+\tfrac{23}{6C}).$
 \end{proof}

 \section{Rigorous  justification of asymptotic models}\label{conv_sec}
 In this section we fully justify the thin-film asymptotic models that we have derived in Section \ref{lubri_sec} for the one-phase unstable gravity-driven Muskat problem in the Wiener framework.
For the sake of clarity, let us denote by $\zeta^{\rm MU}_\mu$ the solution to the evolution equation \eqref{HSeq_dimless} of the one-phase Muskat problem. First, we recall the existence theorem for the solution to the unstable gravity-driven one-phase Muskat problem derived in \cite{GG-BS2020} (Theorem 1.9). We adapt it to the notations we have used throughout this paper.

\begin{theorem}\label{theoMU}
Let $C_0=3120$, $0<\rm Bo,\sqrt{\mu} < 1$ and $T>0$ and let $\zeta_{\mathrm{in},\mu}^{\rm MU}\in\mathbb{A}^1_0$ be a zero-mean initial data such that
\begin{equation}
\label{smallness}|\zeta_{\mathrm{in},\mu}^{\rm MU}|_{1,0}< \min\{\tfrac{1}{C_0 \eps},1\}\left(1-\mathrm{Bo}\right)\sqrt{\mu} ,\end{equation}
then there exists a global solution to \eqref{heleshaw_dimless} with initial data $\zeta_{0,\mu}^{\rm MU}$. Moreover the solution becomes instantaneously analytic in a growing strip in the complex plane. In particular, for any
$$\nu\in \left[0,\tfrac{\sqrt{\mu}}{16}\left(\tfrac{1}{\rm Bo}-1\right)\right),$$
the solution $\zeta^{\rm MU}_\mu$ belongs to 
$$C([0,T]; \mathbb{A}^1_{\nu t})\cap L^1([0,T]; \mathbb{A}^4_{\nu t})$$
for any $T>0.$ Moreover, the following energy inequality holds true for any $t\in[0,T]$
\begin{equation}\label{HS_energy_ineq}|\zeta^{\rm MU}_\mu(t)|_{1,\nu t} + \tfrac{\sqrt{\mu}}{16}\left(\tfrac{1}{\rm Bo}-1\right)\int_{0}^{t} |\zeta^{\rm MU}_\mu(t')|_{4,\nu t'}dt' \leq |\zeta_{\mathrm{in},\mu}^{\rm MU}|_{1,0}\end{equation}
together with the decay 
\begin{equation*}|\zeta^{\rm MU}_\mu|_{1,\nu t} \leq |\zeta_{\mathrm{in},\mu}^{\rm MU}|_{1,0} \  e^{- t \tfrac{\sqrt{\mu}}{16}\left(\frac{1}{\rm Bo} -1\right)}.\end{equation*}
\end{theorem}
We will use later the energy estimate of Theorem \ref{theoMU} to show the asymptotic stability of the thin film approximation.

\subsection{Justification of the thin film approximation at order $O(\mu)$} 
In this subsection we justify in a rigorous way that the thin film equation \eqref{thin-film} is an approximation at order $O(\mu)$ of the one-phase unstable Muskat problem \eqref{HSeq_dimless}. More precisely, we show that the difference between the solution to \eqref{HSeq_dimless} and the solution to \eqref{thin-film} is of order $O(\mu)$ in the Wiener framework.
Let us first denote by $\zeta^{\rm app}$ the solution to the thin-film approximation \eqref{thin-film} and by $f$ the difference $\zeta^{\rm MU}_\mu- \zeta^{\rm app}$. By subtracting \eqref{thin-film} to \eqref{HSeq_dimless}, we obtain that the difference $f$ satisfies the following evolution equation:
 \begin{equation}\label{eq_diff}
 \begin{aligned}
 &\partial_t f + \tfrac{\sqrt{\mu}}{\rm bo}\partial_{xxxx}f+ \partial_{xx}f =-\tfrac{\eps}{\rm Bo}N_1(f,\zeta^{\rm app}) -\eps N_2(f,\zeta^{\rm app}) - \tfrac{1}{\eps} N(\zeta^{\rm MU}_\mu, \widetilde{\phi}^\mu)\\[5pt] 
 \end{aligned}
 \end{equation}
with 
\begin{equation*}
\begin{aligned}
N_1(f,\zeta^{\rm app})= & f\partial_{xxxx} f  +\partial_x f \partial_{xxx}f+  f \partial_{xxxx}\zeta^{\rm app}  + \zeta^{\rm app}\partial_{xxxx} f \\[5pt]&+\partial_x f \partial_{xxx}\zeta^{\rm app}  + \partial_x \zeta^{\rm app}\partial_{xxx}f,\\[10pt]
N_2(f,\zeta^{\rm app})= & f \partial_{xx}f+ (\partial_x f)^2  + f\partial_{xx}\zeta^{\rm app} + 2\partial_x f \partial_x \zeta^{\rm app}+\zeta^{\rm app}\partial_{xx} f
\end{aligned}
\end{equation*}and
\begin{equation*}
N(\zeta^{\rm MU}_\mu, \widetilde{\phi}^\mu)=\partial_x\left(\int_{-1}^0(1+\eps\zeta^{\rm MU}_\mu)\partial_x \widetilde{\phi}^\mu - \eps\partial_x\zeta^{\rm MU}_\mu(z+1) \partial_z\widetilde{\phi}^\mu\right)
\end{equation*}
In order to close energy estimates, we need to control the first two terms in the right hand side of \eqref{eq_diff}. 
\begin{proposition}\label{N1N2_prop}
Let $N_1(f,\zeta^{\rm app})$ and $N_2(f,\zeta^{\rm app})$ be as in \eqref{eq_diff}. The following inequality holds true:
\begin{equation}\label{N1N2_est}
\begin{aligned}
\tfrac{\eps}{\rm Bo}|N_1(f,&\zeta^{\rm app})|_{0,\nu t}+\eps |N_2(f,\zeta^{\rm app})|_{0,\nu t} \\[5pt]&\leq 
2\eps\left(\tfrac{1}{\rm Bo}+ 1 \right)\left(|\zeta^{\rm app}|_{0,\nu t }|f|_{4,\nu t}  +   ( 2|\zeta^{\rm app}|_{4,\nu t }+ |\zeta^{\rm MU}_\mu|_{4,\nu t })|f|_{0,\nu t}\right).
\end{aligned}
\end{equation}
\end{proposition}
\begin{proof}
Using Proposition \ref{productWie} and Proposition \ref{interpolation}, we get
\begin{equation*}
\begin{aligned}
|N_1(f,\zeta^{\rm app})|_{0,\nu t} \leq  \ & \left(2 |f|_{0,\nu t}+ |\zeta^{\rm app}|_{0,\nu t}\right)|f|_{4,\nu t} + |f|_{0,\nu t}  |\zeta^{\rm app}|_{4,\nu t}\\[5pt]&+ |\zeta^{\rm app}|_{1,\nu t}|f|_{3,\nu t}+|\zeta^{\rm app}|_{3,\nu t}|f|_{1,\nu t}
\end{aligned}
\end{equation*}and
\begin{equation*}
\begin{aligned}
|N_2(f,\zeta^{\rm app})|_{0,\nu t} \leq  \  \left(2 |f|_{0,\nu t}+ |\zeta^{\rm app}|_{0,\nu t}\right)|f|_{2,\nu t} +  |f|_{0,\nu t}|\zeta^{\rm app}|_{2,\nu t} + 2 |\zeta^{\rm app}|_{1,\nu t}|f|_{1,\nu t}.
\end{aligned}
\end{equation*}
Using again Proposition \ref{interpolation} and Young's inequality, we have that 

\begin{equation}\begin{aligned}
 |\zeta^{\rm app}|_{1,\nu t}|f|_{3,\nu t}+&|\zeta^{\rm app}|_{3,\nu t}|f|_{1,\nu t} +	  2 |\zeta^{\rm app}|_{1,\nu t}|f|_{1,\nu t}\\[5pt]& \leq 2 \left( |\zeta^{\rm app}|_{4,\nu t }|f|_{0,\nu t} + |\zeta^{\rm app}|_{0,\nu t }|f|_{4,\nu t} \right),
 \end{aligned}
\end{equation} which gives the following control on the first two terms of the right-hand side of \eqref{eq_diff}:
\begin{equation*}
\begin{aligned}
&\tfrac{\eps}{\rm Bo}|N_1(f,\zeta^{\rm app})|_{0,\nu t} +\eps |N_2(f,\zeta^{\rm app})|_{0,\nu t} \\[5pt]&\leq 2\eps\left(\tfrac{1}{\rm Bo} + 1 \right)\left(|\zeta^{\rm app}|_{0,\nu t } |f|_{4,\nu t} +  (|f|_{4,\nu t}+|\zeta^{\rm app}|_{4,\nu t })|f|_{0,\nu t}\right)\\[5pt]&\leq 
2\eps\left(\tfrac{1}{\rm Bo}+ 1 \right)\left(|\zeta^{\rm app}|_{0,\nu t }|f|_{4,\nu t}  +   ( 2|\zeta^{\rm app}|_{4,\nu t }+ |\zeta^{\rm MU}_\mu|_{4,\nu t })|f|_{0,\nu t}\right), 
\end{aligned}
\end{equation*}where we have used the fact that  $f=\zeta^{\rm MU}_\mu - \zeta^{\rm app}_{\rm ref}.$
\end{proof}
We focus now on the term $N(\zeta^{\rm MU}_\mu, \widetilde{\phi}^\mu)$. Our goal is to show that its Wiener norm can be controlled by $\mu$ multiplied by a function uniformly bounded with respect to $\mu$ in $L^1([0,T])$. To do that, we first show a Poincaré-like inequality for functions in $\mathcal{A}^{s,1}_\lambda$.
\begin{lemma}\label{poincare}
	Let $f\in \mathcal{A}^{s,1}_\lambda$ for $s,\lambda\geq 0$. Then $\|f- f_{|_{z=-1,0}}\|_{\mathcal{A}^{s,0}_\lambda}\leq \|\partial_zf\|_{\mathcal{A}^{s,0}_\lambda}$.
\end{lemma}
\begin{proof}
 For $f\in \mathcal{A}^{s,1}_\lambda$ and $z\in[-1,0]$, we have 
	\begin{equation*}\begin{aligned}
	|\hat{f}(n,z)- \hat{f}(n,-1)|=\left|\int_{-1}^z \partial_r \hat{f}(n,r)dr\right| \leq \int_{-1}^z|\partial_r \hat{f}(n,r) |dr \leq \int_{-1}^0|\partial_z \hat{f}(n,z)|dz,
	\end{aligned}
	\end{equation*}from which it follows that
$	\|f- f_{|_{z=-1}}\|_{\mathcal{A}^{s,0}_\lambda}\leq \|\partial_z f\|_{\mathcal{A}^{s,0}_\lambda}$. The same estimate holds replacing $f_{|_{z=-1}}$ with $f_{|_{z=0}}$. 
\end{proof}
 \begin{proposition}\label{propNphimu}
Let $\mu<1$ and $N(\zeta^{\rm MU}_\mu, \widetilde{\phi}^\mu)$ be as in \eqref{eq_diff} with $\widetilde{\phi}^\mu$ solution to \eqref{laplace_phimu2} and $\zeta^{\rm MU}_\mu$ solution to \eqref{HSeq_dimless}. Then, the following inequality holds:
\begin{equation}\label{Nphimu_est}
	\tfrac{1}{\eps}|N(\zeta^{\rm MU}_\mu, \widetilde{\phi}^\mu)|_{0,\nu t}\leq \mu  R^\mu,
\end{equation}
for some time-dependent function $R^\mu.$ Moreover, if the initial data $\zeta^{\rm MU}_{\mathrm{in},\mu}$  belongs to $\mathbb{A}^{3}_0$ then $ R^\mu$ is uniformly bounded  in $L^1([0,T])$ with respect of $\mu$ for any $T>0$.
 \end{proposition}
\begin{proof}
From the definition of $N(\zeta^{\rm MU}_\mu, \widetilde{\phi}^\mu)$ we have
\begin{equation}\label{estNphimu}
\begin{aligned}
&|N(\zeta^{\rm MU}_\mu, \widetilde{\phi}^\mu)|_{0,\nu t}\\&\leq \left|\partial_x\left((1+\eps\zeta^{\rm MU}_{\mu})\int_{-1}^0 \partial_x \tilde{\phi}^\mu dz\right)\right|_{0,\nu t}  +\left|\eps\partial_x\zeta^{\rm MU}_{\mu}\int_{-1}^{0}(z+1) \partial_z\widetilde{\phi}^\mu dz\right|_{1,\nu t}.
\end{aligned}
\end{equation}
Using Proposition \ref{productWie}, we estimate the first term in the right-hand side as  
\begin{equation*}
\begin{aligned}
&\left|\partial_x\left((1+\eps\zeta^{\rm MU}_{\mu})\int_{-1}^0 \partial_x \tilde{\phi}^\mu dz\right)\right|_{0,\nu t} \\&\leq   (1+\eps|\zeta^{\rm MU}_{\mu}|_{0,\nu t})\left(\left|\int_{-1}^0 (\partial_x \tilde{\phi}^\mu -\partial_x \tilde{\phi}^\mu_{|_{z=0}})dz\right|_{1,\nu t} + | \partial_x\tilde{\phi}^\mu_{|_{z=0}}|_{1,\nu t}\right)\\&\qquad + \eps|\zeta^{\rm MU}_{\mu}|_{1,\nu t}\left(\left|\int_{-1}^0 (\partial_x \tilde{\phi}^\mu -\partial_x \tilde{\phi}^\mu_{|_{z=0}})dz\right|_{0,\nu t} + | \partial_x\tilde{\phi}^\mu_{|_{z=0}}|_{0,\nu t}\right)\\&
\leq   (1+\eps|\zeta^{\rm MU}_{\mu}|_{0,\nu t})(\|\partial_x \tilde{\phi}^\mu -\partial_x \tilde{\phi}^\mu_{|_{z=0}}\|_{\mathcal{A}^{1,0}_{\nu t}} + | \tilde{\phi}^\mu_{|_{z=0}}|_{2,\nu t})\\&\qquad+ \eps|\zeta^{\rm MU}_{\mu}|_{1,\nu t}(\|\partial_x \tilde{\phi}^\mu -\partial_x \tilde{\phi}^\mu_{|_{z=0}}\|_{\mathcal{A}^{0,0}_{\nu t}} + | \tilde{\phi}^\mu_{|_{z=0}}|_{1,\nu t})
\end{aligned}
\end{equation*}
and the second term as
\begin{equation*}
\begin{aligned}
&\left|\eps\partial_x\zeta^{\rm MU}_{\mu} \int_{-1}^0(z+1)\partial_z \widetilde{\phi}^\mu dz\right|_{1,\nu t} \\&\leq 
\eps|\zeta^{\rm MU}_{\mu}|_{1,\nu t}\left|\int_{-1}^0(z+1) \partial_z \widetilde{\phi}^\mu dz\right|_{1,\nu t}+ \eps|\zeta^{\rm MU}_{\mu}|_{2,\nu t}\left|\int_{-1}^0(z+1) \partial_z \widetilde{\phi}^\mu dz\right|_{0,\nu t}\\&\leq 
\eps|\zeta^{\rm MU}_{\mu}|_{1,\nu t}\|\partial_z \widetilde{\phi}^\mu\|_{\mathcal{A}^{1,0}_{\nu t}} + \eps|\zeta^{\rm MU}_{\mu}|_{2,\nu t}\|\partial_z \widetilde{\phi}^\mu\|_{\mathcal{A}^{0,0}_{\nu t}}.
\end{aligned}
\end{equation*}
Lemma \ref{poincare} permits to obtain a control of the $x$-derivative of the potential remainder with the dimensionless gradient $\nabla^\mu \widetilde{\phi}^\mu$ without the appearance of the singular term $(\sqrt{\mu})^{-1}$ but at the cost of increasing the regularity of the potential remainder in the $x$-direction. More precisely, for $i=0,1$, $$\|\partial_x \widetilde{\phi}^\mu\ - \partial_x \widetilde{\phi}_{|_{z=0}}^\mu\|_{\mathcal{A}^{i,0}_{\nu t}}\leq \|\partial_z\partial_x \widetilde{\phi}^\mu\|_{\mathcal{A}^{i,0}_{\nu t}}\leq \|\partial_z \widetilde{\phi}^\mu\|_{\mathcal{A}^{i+1,0}_{\nu t}}\leq  \|\nabla^{\mu} \widetilde{\phi}^\mu\|_{\mathcal{A}^{i+1,0}_{\nu t}},$$
Then, from \eqref{estNphimu} we obtain 
\begin{equation}\label{estNmu_poincare}
\begin{aligned}
|N(\zeta^{\rm MU}_\mu, &\widetilde{\phi}^\mu)|_{0,\nu t}\leq \ (1+\eps |\zeta^{\rm MU}_{\mu}|_{0,\nu t})(\|\nabla^\mu \widetilde{\phi}^\mu\|_{\mathcal{A}^{2,0}_{\nu t}}+ |\widetilde{\phi}^\mu_{|_{z=0}}|_{2,\nu t}) \\[5pt]&+ \eps |\zeta^{\rm MU}_{\mu}|_{1,\lambda}(2\|\nabla^\mu \widetilde{\phi}^\mu\|_{\mathcal{A}^{1,0}_{\nu t}}+ |\widetilde{\phi}^\mu_{|_{z=0}}|_{1,\nu t})+\eps |\zeta^{\rm MU}_{\mu}|_{2,\nu t} \|\nabla^\mu\widetilde{\phi}^\mu\|_{\mathcal{A}^{0,0}_{\nu t}}.
\end{aligned}
\end{equation}From \eqref{laplace_phimu} we have $\widetilde{\phi}^\mu_{|_{z=0}}=\tfrac{\eps}{\rm Bo}F_{\eps\sqrt{\mu}}(\partial_x \zeta)\partial_{xx}\zeta$. Then, Proposition \ref{productWie}, Proposition \ref{interpolation} and 
Proposition \ref{ell_est}  yield
\begin{equation*}
\begin{aligned}
|N(\zeta^{\rm MU}_\mu, \widetilde{\phi}^\mu)|_{0,\nu t} \leq   8C\eps \mu (1+\tfrac{1}{\rm Bo}) (1+ 5\eps|\zeta^{\rm MU}_\mu|_{0,\nu t}+ 2\eps^2|\zeta^{\rm MU}_\mu|^2_{0,\nu t})|\zeta^{\rm MU}_\mu|_{6,\nu t}
\end{aligned}
\end{equation*}
and, by defining
\begin{equation}\label{Rmu}
	R^\mu=8 C \left(1+\tfrac{1}{\rm Bo}\right)\left(1+ 5\eps|\zeta^{\rm MU}_{\mu}|_{0,\nu t} + 2\eps^2|\zeta^{\rm MU}_\mu|^2_{0,\nu t} \right) |\zeta^{\rm MU}_{\mu}|_{6,\nu t},
\end{equation} the desired estimate \eqref{Nphimu_est} follows. Let us remark that Theorem \ref{theoMU} gives a minimal regularity result that stating the uniform boundedness of $\zeta^{\rm MU}_{\mu}\in L^\infty(0,T;\mathbb{A}^1_{\nu t})\cap L^1(0,T;\mathbb{A}^4_{\nu t})$ with respect to $\mu$ for any $T>0$. More precisely, due to the smallness condition \eqref{smallness} and since $\sqrt{\mu}<1$ the following bounds hold for any $t\in[0,T]$
\begin{equation*}
	|\zeta^{\rm MU}_{\mu}(t)|_{1,\nu t}\leq	|\zeta^{\rm MU}_{\mathrm{in}, \mu}|_{1,0}< \min\{\tfrac{1}{C_0\eps},1\}(1-\mathrm{Bo}),
\end{equation*}and 
\begin{equation*}
\int_0^t|\zeta^{\rm MU}_{\mu}(t')|_{4,\nu t'}dt'< 16 \min\{\tfrac{1}{C_0\eps},1\}\mathrm{Bo}.
\end{equation*}
 Analogously, the latter result can be stated with higher regularity. In particular, for initial data in $\mathbb{A}^3_0$ we obtain that $\zeta^{\rm MU}_{\mu}$ is uniformly bounded in $ L^\infty(0,T;\mathbb{A}^3_{\nu t})\cap L^1(0,T;\mathbb{A}^6_{\nu t})$ with respect to $\mu$ for any $T>0$. Therefore, since $\mathbb{A}^3_0 \subset\mathbb{A}^0_0 $, 
  $R^\mu$ defined in \eqref{Rmu} is uniformly bounded in $L^1([0,T])$ with respect to $\mu$ provided that the initial data belong to $\mathbb{A}^3_0$.
 \end{proof}
 We are now able to state the following asymptotic stability theorem for the thin-film equation \eqref{thin-film}: 
 \begin{theorem}\label{conv_theo}
Let $0<\mathrm{Bo} < 1$ and $\zeta_{\mathrm{in},\mu}$ be a zero-mean function in $\mathbb{A}^3_0$ satisfying the smallness assumption \eqref{smallness}. Let $\zeta^{\rm MU}_\mu$  and $\zeta^{\rm app}$ be the global weak solutions respectively to \eqref{heleshaw_dimless} and to \eqref{thin-film} both with initial data $\zeta_{\mathrm{in},\mu}.$ Then, for any $T>0$ there exists $C>0$ independent of $\mu$ such that for $\mu\in(0,1)$
\begin{equation}\label{stability}
	\|\zeta^{\rm MU}_\mu -\zeta^{\rm app}\|_{L^\infty([0,T]; \ \mathbb{A}^0_{\nu t})}+ \tfrac{1}{64}(\tfrac{1}{\rm Bo}-1)	\|\zeta^{\rm MU}_\mu -\zeta^{\rm app}\|_{L^1([0,T]; \ \mathbb{A}^4_{\nu t})}\leq  C\mu.
\end{equation}
 \end{theorem}
\vspace{0.1em}
 \begin{proof}
Analogously to what done in Section \ref{lubri_sec} for \eqref{apriori}, we implement an energy estimate for \eqref{eq_diff} and we get for $f=\zeta^{\rm MU}_\mu -\zeta^{\rm app}$ the following inequality:
 \begin{equation*}\begin{aligned}
\frac{d}{dt} | f|_{0,\nu t} &+\tfrac{1}{32}(\tfrac{1}{\rm Bo}-1)|f|_{4,\nu t}\\[5pt]&\leq \tfrac{\eps}{\rm Bo}|N_1(f,\zeta^{\rm app})|_{0,\nu t} +\eps |N_2(f,\zeta^{\rm app})|_{0,\nu t} + \tfrac{1}{\eps}|N(\zeta^{\rm MU}_\mu, \widetilde{\phi}^\mu)|_{0,\nu t}.
 \end{aligned}
 \label{apriori-diff}
 \end{equation*}
 From \eqref{N1N2_est} and \eqref{Nphimu_est}, it yields
 \begin{equation}\begin{aligned}
 	&\frac{d}{dt} | f|_{0,\nu t} +\tfrac{1}{32}(\tfrac{1}{\rm Bo}-1)|f|_{4,\nu t}\\[5pt]&\leq 2\eps\left(\tfrac{1}{\rm Bo}+ 1 \right)\left(|\zeta^{\rm app}|_{0,\nu t }|f|_{4,\nu t}    + ( 2|\zeta^{\rm app}|_{4,\nu t }+ |\zeta^{\rm MU}_\mu|_{4,\nu t })|f|_{0,\nu t}\right)+ \mu R^\mu
 	\end{aligned}
 \end{equation}
 and using the estimate \eqref{energy_estimate} for $\zeta^{\rm app}$ together with the smallness assumption \eqref{smallnessTF} we get
 \begin{equation*}
 \frac{d}{dt} | f|_{0,\nu t} +\tfrac{1}{64}(\tfrac{1}{\rm Bo}-1)|f|_{4,\nu t}\leq 2\eps\left(\tfrac{1}{\rm Bo}+ 1 \right)\left( 2|\zeta^{\rm app}|_{4,\nu t }+ |\zeta^{\rm MU}_\mu|_{4,\nu t }\right)|f|_{0,\nu t}+ \mu R^\mu
 \end{equation*}
From Theorem \ref{theoTF} and Theorem \ref{theoMU} we respectively know that $\zeta^{\rm app}$ and $(\zeta^{\rm MU}_\mu)_{0<\mu<1}$ are uniformly bounded in $L^1([0,T];\mathbb{A}^4_{\nu t})$. From Proposition \ref{propNphimu} we know that if $\zeta_{\mathrm{in},\mu}\in  \mathbb{A}^3_0$ then $(R^\mu)_\mu$ is uniformly bounded in $L^1([0,T])$. Therefore, using Gronwall's inequality we obtain that for all $0<\mu<1$

\begin{equation*}\begin{aligned}
&	|f(t)|_{0,\nu t}  \ +  \  \tfrac{1}{64}\left(\tfrac{1}{\rm Bo}-1\right)\int_{0}^{t} |f(t')|_{4,\nu t'} dt' \\&\leq\mu\left( \int_{0}^{t} R^\mu(t')dt' \right)e^{ 2\eps \left(\tfrac{\sqrt{\mu}}{\rm bo}+1\right) \int_{0}^{t}\left(2|\zeta^{\rm app}(t')|_{4,\nu t' }+ |\zeta^{\rm MU}_\mu(t')|_{4,\nu t '}\right)dt'}\leq C \mu 
	\end{aligned}
\end{equation*}
for all $t\in[0,T]$ and some constant $C>0$ independent of $\mu$. Notice that in the previous inequality we have used that $f(0)=0$. The error estimate \eqref{stability} easily follows after taking the supremum over $[0,T]$.
 \end{proof}
\begin{remark}
Since $\|f\|_{L^\infty(\mathbb{T})}\leq |f|_{0,\nu t }$, it follows from \eqref{stability} that for all $T>0$ and $\mu\in (0,1)$ we have  $$	\|\zeta^{\rm MU}_\mu -\zeta^{\rm app}\|_{L^\infty([0,T]\times \mathbb{T})}\leq   C\mu.$$
\end{remark}

  \subsection{Justification of the thin film approximation at order $O(\mu^{3/2})$}
  In this subsection, analogously to what done in the previous one, we justify in a rigorous way that the refined thin film equation \eqref{O(mu)} is an approximation at order $O(\mu^{3/2})$ of the one-phase unstable Muskat problem. More precisely, we show that the difference between the solution to \eqref{HSeq_dimless} and the solution to \eqref{O(mu)} is of order $O(\mu^{3/2})$ in the Wiener framework. Let us denote now by $\zeta^{\rm app}_{\rm ref}$ the solution to \eqref{O(mu)} and by $f$ the difference  $\zeta^{\rm MU}_\mu-\zeta^\mathrm{app}_{\rm ref}$. Then $f$ satisfies the following evolution equation:
  \begin{equation}\label{eq_diff_ref}
  	\begin{aligned}
  &	\partial_t f + \left( \tfrac{\sqrt{\mu}}{\rm bo}  +\tfrac{\mu}{3}\right)\partial_{xxxx}f + \partial_{xx}f\\[5pt]&
  = -\tfrac{\eps\sqrt{\mu}}{\rm bo}N_1(f,\zeta^{\rm app}_{\rm ref}) -\eps N_2(f,\zeta^{\rm app}_{\rm ref})  -\eps\mu N_3(f,\zeta^{\rm app}_{\rm ref}) -\tfrac{1}{\eps}N(\zeta^{\rm MU}_\mu, \widetilde{\phi}^\mu_{\rm ref})\\[5pt]
  	\end{aligned}
  \end{equation}
 with $N_1(\cdot,\cdot)$, $N_2(\cdot,\cdot)$, $N(\zeta^{\rm MU}_\mu, \cdot)$ as in \eqref{eq_diff} and 
 \begin{equation*}
 \begin{aligned}
 	&N_3(f,\zeta^{\rm app}_{\rm ref})
 	\\[5pt]&=\partial_{xx}\big[  f\partial_{xx}f + \zeta^{\rm app}_{\rm ref}\partial_{xx}f + f\partial_{xx}\zeta^{\rm app}_{\rm ref} \\[5pt]&\qquad\quad  + \eps\left(f^2\partial_{xx}f + (\zeta^{\rm app}_{\rm ref})^2\partial_{xx}f+ f^2\partial_{xx}\zeta^{\rm app}_{\rm ref} + 2f\zeta^{\rm app}_{\rm ref}\partial_{xx}f +2f\zeta^{\rm app}_{\rm ref}\partial_{xx}\zeta^{\rm app}_{\rm ref}  \right)\\[5pt]& \qquad\quad+ \tfrac{\eps^2}{3}\big(f^3\partial_{xx}f + f^3\partial_{xx}\zeta^{\rm app}_{\rm ref} + 3f^2\zeta^{\rm app}_{\rm ref}\partial_{xx}f + 3f^2 \zeta^{\rm app}_{\rm ref}\partial_{xx}\zeta^{\rm app}_{\rm ref}\\[5pt]&\qquad\qquad\quad  + 3 f (\zeta^{\rm app}_{\rm ref})^2\partial_{xx}f + 3 f (\zeta^{\rm app}_{\rm ref})^2\partial_{xx}\zeta^{\rm app}_{\rm ref} + (\zeta^{\rm app}_{\rm ref})^3\partial_{xx}f \big)  \big].
 	\end{aligned}
 \end{equation*}
 The first two terms in the right hand side of \eqref{eq_diff_ref} can be estimated in the same way as in \eqref{N1N2_est}. Therefore let us focus on $N_3(f,\zeta^{\rm app}_{\rm ref})$ and $N(\zeta^{\rm MU}_\mu, \widetilde{\phi}^\mu_{\rm ref})$, whose estimates in the Wiener framework are given in the next two propositions.
 
 \begin{proposition}Let $N_3(f,\zeta^{\rm app}_{\rm ref})$ be as in \eqref{eq_diff_ref}. The following estimate holds true:
 	\begin{equation}\label{N3_est}
 	\begin{aligned}
 		\eps\mu|N_3(f,\zeta^{\rm app}_{\rm ref})|_{0,\nu t} \leq \mu F_1(\zeta^{\rm app}_{\rm ref},\zeta^{\rm MU}_{\mu} ) |f|_{4,\nu t} +\mu F_2(\zeta^{\rm app}_{\rm ref},\zeta^{\rm MU}_{\mu} )|\zeta^{\rm app}_{\rm ref}|_{4,\nu t} |f|_{0,\nu t}
 			\end{aligned}
 				\end{equation}
 				with 
 					\begin{equation*}
 				\begin{aligned}
 	F_1(\zeta^{\rm app}_{\rm ref},\zeta^{\rm MU}_{\mu} )
 	= C_1\eps(|\zeta^{\rm app}_{\rm ref}|_{0,\nu t }&+ |\zeta^{\rm MU}_{\mu}|_{0,\nu t })\\[5pt]
 	&\times\left( 1+ \eps |\zeta^{\rm app}_{\rm ref}|_{0,\nu t} + \eps^2(|\zeta^{\rm app}_{\rm ref}|^2_{0,\nu t } + |\zeta^{\rm MU}_{\mu}|^2_{0,\nu t }) \right)
 	\end{aligned}
 \end{equation*}
 and 
 
 	\begin{equation*}
 \begin{aligned}
F_2(\zeta^{\rm app}_{\rm ref},\zeta^{\rm MU}_{\mu} )= \ C_2\eps \Big(1&+\eps(|\zeta^{\rm app}_{\rm ref}|_{0,\nu t } + |\zeta^{\rm MU}_{\mu}|_{0,\nu t })  \\[5pt]
&+\eps^2(|\zeta^{\rm app}_{\rm ref}|^2_{0,\nu t } +  |\zeta^{\rm app}_{\rm ref}|_{0,\nu t }|\zeta^{\rm MU}_{\mu}|_{0,\nu t }+ |\zeta^{\rm MU}_{\mu}|^2_{0,\nu t })\Big),\\[5pt]
 \end{aligned}
 \end{equation*} where $C_1$ and $C_2$ are two positive constants independent of $\mu$. 
 \end{proposition}
 \begin{proof}
The proof follows exactly the same argument used in the proof of Proposition \ref{N1N2_prop}. For the sake of clarity, we only show here how to estimate some terms of $N_3(f,\zeta^{\rm app}_{\rm ref})$ and the rest of the terms can be estimated in the same way. First, we consider the terms where $f$ interacts with itself.
Using Proposition \ref{productWie}, Proposition \ref{interpolation} and the fact that $f=\zeta^{\rm MU}_\mu - \zeta^{\rm app}_{\rm ref}$ one obtains
\begin{equation*}
	\begin{aligned}
|\partial_{xx}(f\partial_{xx}f)|_{0,\nu t} &\leq |f\partial_{xx}f|_{2,\nu t} \leq K_2(|f|^2_{2,\nu t} + |f|_{0,\nu t}|f|_{4,\nu t})\\[5pt]&\leq 2K_2 |f|_{0,\nu t}|f|_{4,\nu t} \leq 2K_2 \left(|\zeta^{\rm app}_{\rm ref}|_{0,\nu t}+|\zeta^{\rm MU}_\mu|_{0,\nu t}\right)|f|_{4,\nu t}.\\[5pt]
	\end{aligned}
\end{equation*}Moreover, using also Proposition \ref{powerWie} we have 
\begin{equation*}
\begin{aligned}
|\partial_{xx}(f^2\partial_{xx}f)|_{0,\nu t} &\leq |f^2\partial_{xx}f|_{2,\nu t} \leq K_2( |f|^2_{0,\nu t}|f|_{4,\nu t}+ |f^2|_{2,\nu t}|f|_{2,\nu t})\\[5pt]&\leq K_2 \left( |f|^2_{0,\nu t}|f|_{4,\nu t} +K_{2,2}|f|_{0,\nu t} |f|^2_{2,\nu t} \right)\leq K_2(1+K_{2,2})|f|^2_{0,\nu t}|f|_{4,\nu t}\\[5pt]&
\leq K_2(1+K_{2,2}) \left(|\zeta^{\rm app}_{\rm ref}|_{0,\nu t}+|\zeta^{\rm MU}_\mu|_{0,\nu t}\right)^2|f|_{4,\nu t}
\end{aligned}
\end{equation*}
and 
\begin{equation*}
\begin{aligned}
|\partial_{xx}(f^3\partial_{xx}f)|_{0,\nu t} &\leq |f^3\partial_{xx}f|_{2,\nu t} \leq K_2( |f|^3_{0,\nu t}|f|_{4,\nu t}+|f^3|_{2,\nu t}|f||_{2,\nu t})\\[5pt]&\leq K_2 \left( |f|^3_{0,\nu t}|f|_{4,\nu t} +K_{2,3}|f|^2_{0,\nu t} |f|^2_{2,\nu t} \right)\leq K_2(1+K_{2,3})|f|^3_{0,\nu t}|f|_{4,\nu t}\\[5pt]&
\leq K_2(1+K_{2,3}) \left(|\zeta^{\rm app}_{\rm ref}|_{0,\nu t}+|\zeta^{\rm MU}_\mu|_{0,\nu t}\right)^3|f|_{4,\nu t}
\end{aligned}
\end{equation*}
Secondly, we consider two mixed terms. We have

\begin{equation*}
\begin{aligned}
|\partial_{xx}(\zeta^{\rm app}_{\rm ref}\partial_{xx}f &+f\partial_{xx} \zeta^{\rm app}_{\rm ref})|_{0,\nu t} \leq |\zeta^{\rm app}_{\rm ref}\partial_{xx}f|_{2,\nu t} +|f\partial_{xx} \zeta^{\rm app}_{\rm ref}|_{2,\nu t}\\[5pt]& \leq K_2 \left(|\zeta^{\rm app}_{\rm ref}|_{0,\nu t}|f|_{4,\nu t}+ |\zeta^{\rm app}_{\rm ref}|_{4,\nu t}|f|_{0,\nu t} + 2 |\zeta^{\rm app}_{\rm ref}|_{2,\nu t}|f|_{2,\nu t}\right)\\[5pt]& \leq 2K_2\left(\zeta^{\rm app}_{\rm ref}|_{0,\nu t}|f|_{4,\nu t}+ |\zeta^{\rm app}_{\rm ref}|_{4,\nu t}|f|_{0,\nu t} \right)\\[5pt]
\end{aligned}
\end{equation*}
where we have applied Young's inequality to obtain
\begin{equation*}
	\begin{aligned}
|\zeta^{\rm app}_{\rm ref}|_{2,\nu t}|f|_{2,\nu t}&\leq|\zeta^{\rm app}_{\rm ref}|^{1/2}_{0,\nu t}|\zeta^{\rm app}_{\rm ref}|^{1/2}_{4,\nu t}|f|^{1/2}_{0,\nu t}|f|^{1/2}_{4,\nu t}\\[5pt]& \leq \tfrac{1}{2}|\zeta^{\rm app}_{\rm ref}|_{4,\nu t}|f|_{0,\nu t} +\tfrac{1}{2}|\zeta^{\rm app}_{\rm ref}|_{0,\nu t}|f|_{4,\nu t}.
	\end{aligned}
\end{equation*} Using the same inequality we are able to estimate all the mixed terms. For instance, we get

\begin{equation*}
\begin{aligned}
|&\partial_{xx}((\zeta^{\rm app}_{\rm ref})^2\partial_{xx}f +f^2\partial_{xx} \zeta^{\rm app}_{\rm ref})|_{0,\nu t} \leq |(\zeta^{\rm app}_{\rm ref})^2\partial_{xx}f|_{2,\nu t} +|f^2\partial_{xx} \zeta^{\rm app}_{\rm ref}|_{2,\nu t}\\[5pt]& \leq K_2 \left(|\zeta^{\rm app}_{\rm ref}|^2_{0,\nu t}|f|_{4,\nu t}+ |\zeta^{\rm app}_{\rm ref}|_{4,\nu t}|f|^2_{0,\nu t} + K_{2,2} |\zeta^{\rm app}_{\rm ref}|_{2,\nu t}|f|_{2,\nu t}(|f|_{0,\nu t}+|\zeta^{\rm app}_{\rm ref}|_{0,\nu t})\right)
\\[5pt]& \leq \left((K_2 +\tfrac{K_2K_{2,2}}{2})|\zeta^{\rm app}_{\rm ref}|_{0,\nu t} + \tfrac{K_2K_{2,2}}{2}|f|_{0,\nu t}\right)|\zeta^{\rm app}_{\rm ref}|_{0,\nu t }|f|_{4,\nu t}\\[5pt] &
\quad +  \left((K_2 +\tfrac{K_2K_{2,2}}{2})|f|_{0,\nu t} + \tfrac{K_2K_{2,2}}{2}|\zeta^{\rm app}_{\rm ref}|_{0,\nu t}\right)|\zeta^{\rm app}_{\rm ref}|_{4,\nu t }|f|_{0,\nu t}
\\[5pt]&\leq  \left(K_2 (1+K_{2,2})|\zeta^{\rm app}_{\rm ref}|^2_{0,\nu t} + \tfrac{K_2K_{2,2}}{2}|\zeta^{\rm app}_{\rm ref}|_{0,\nu t}|\zeta^{\rm MU}_\mu|_{0,\nu t}\right)||f|_{4,\nu t}\\[5pt] &
\quad + \left(K_2 (1+K_{2,2})|\zeta^{\rm app}_{\rm ref}|_{0,\nu t} + \tfrac{K_2K_{2,2}}{2}|\zeta^{\rm MU}_\mu|_{0,\nu t}\right)||\zeta^{\rm app}_{\rm ref}|_{4,\nu t}|f|_{0,\nu t}.
\end{aligned}
\end{equation*}
Therefore, gathering together the estimates for all the terms of $N_3(f,\zeta^{\rm app}_{\rm ref})$ and computing the explicit constants $K_2$, $K_{2,2}$ and $K_{2,3}$, the estimate \eqref{N3_est} follows.
 \end{proof}

 \begin{proposition}\label{propNphimu_ref}
 Let $\mu<1$ and $N(\zeta^{\rm MU}_\mu, \widetilde{\phi}^\mu_{\rm ref})$ be as in \eqref{eq_diff_ref} with $\widetilde{\phi}^\mu_{\rm ref}$ solution to \eqref{laplace_phimu_ref} and $\zeta^{\rm MU}_\mu$ solution to \eqref{HSeq_dimless}. Then, the following inequality holds:
 \begin{equation}\label{Nphimu_ref_est}
 \tfrac{1}{\eps}|N(\zeta^{\rm MU}_\mu, \widetilde{\phi}^\mu_{\rm ref})|_{0,\nu t}\leq \mu^{3/2}  R^\mu_{\rm ref},
 \end{equation}
 for some time-dependent function $R^\mu_{\rm ref}$. Moreover,  if the initial data $\zeta^{\rm MU}_{\mathrm{in},\mu}$  belongs to $\mathbb{A}^{3}_0$ then $ R^\mu$ is uniformly bounded  in $L^1([0,T])$ with respect of $\mu$ for any $T>0$.
\end{proposition}
\begin{proof}
The proof is essentially the same as the proof of Proposition \ref{propNphimu}. We have 
\begin{equation}\label{estNmuref_poincare}
\begin{aligned}
&|N(\zeta^{\rm MU}_\mu, \widetilde{\phi}^\mu_{\rm ref})|_{0,\nu t}\leq \ (1+\eps |\zeta^{\rm MU}_{\mu}|_{0,\nu t})(\|\nabla^\mu \widetilde{\phi}^\mu_{\rm ref}\|_{\mathcal{A}^{2,0}_{\nu t}}+ |(\widetilde{\phi}^\mu_{\rm ref})_{|_{z=0}}|_{2,\nu t}) \\[2pt]&+ \eps |\zeta^{\rm MU}_{\mu}|_{1,\lambda}(2\|\nabla^\mu \widetilde{\phi}^\mu_{\rm ref}\|_{\mathcal{A}^{1,0}_{\nu t}}+ |(\widetilde{\phi}^\mu_{\rm ref})_{|_{z=0}}|_{1,\nu t})+\eps |\zeta^{\rm MU}_{\mu}|_{2,\nu t} \|\nabla^\mu\widetilde{\phi}^\mu_{\rm ref}\|_{\mathcal{A}^{0,0}_{\nu t}}
\end{aligned}
\end{equation}
with $\widetilde{\phi}^\mu_{\rm ref}$ solution to \eqref{laplace_phimu_ref}. Using Proposition \ref{productWie}, Proposition \ref{interpolation} and Proposition \ref{ell_est_ref}, there exist $K>0$ independent of $\mu$ and a degree-two polynomial $P(\cdot)$ such that 
\begin{equation*}\begin{aligned}
|N(\zeta^{\rm MU}_\mu, \widetilde{\phi}^\mu_{\rm ref})|_{0,\nu t}
\leq & K \eps \mu^{3/2} \left(1+\tfrac{1}{\rm bo}\right)\left(1 + P(\eps |\zeta^{\rm MU}_{\mu}|_{0,\nu t}) \right)|\zeta^{\rm MU}_{\mu}|_{6,\nu t},\\[5pt]
\end{aligned}
\end{equation*}
and, by defining 
\begin{equation*}
R^\mu_{\rm ref}=K\left(1+\tfrac{1}{\rm bo}\right)\left(1+ P(\eps|\zeta^{\rm MU}_{\mu}|_{0,\nu t}) \right)|\zeta^{\rm MU}_{\mu}|_{6,\nu t},\\[2pt]
\end{equation*} the desired estimate \eqref{Nphimu_ref_est} follows. From the higher regularity version of Theorem \ref{theoMU}, $R^\mu_\mathrm{ref}$ is uniformly bounded in $L^1([0,T])$ with respect to $\mu$ for any $T>0$ provided the initial data are in $\mathbb{A}^3_0$.
\end{proof}
We are now able to state the following asymptotic stability theorem for the refined thin-film equation \eqref{O(mu)}:
\begin{theorem}\label{conv_ref_theo}
	Let $\mathrm{bo>0}$ small enough and $\zeta_{\mathrm{in},\mu}$ be a zero-mean function in $\mathbb{A}^3_0$ satisfying the smallness assumption \eqref{smallness}. Let $\zeta^{\rm MU}_\mu$  and $\zeta^{\rm app}_{\rm ref}$ be the global weak solutions respectively to \eqref{heleshaw_dimless} and to \eqref{O(mu)} both with initial data $\zeta_{\mathrm{in},\mu}.$ Then, for any $T>0$ there exist some constant $C>0$ independent of $\mu$ and $\mu_0\in (0,1)$ such that for $\mu\in [ \mu_0,1)$
	\begin{equation}\label{stability_ref}
	\|\zeta^{\rm MU}_\mu -\zeta^{\rm app}_{\rm ref}\|_{L^\infty([0,T]; \ \mathbb{A}^0_{\nu t})}+ \tfrac{1}{64}(\tfrac{\sqrt{\mu}}{\rm bo}+\tfrac{\mu}{3}-1)	\|\zeta^{\rm MU}_\mu -\zeta^{\rm app}_{\rm ref}\|_{L^1([0,T]; \ \mathbb{A}^4_{\nu t})}\leq  C\mu^{3/2}.
	\end{equation}
\end{theorem}
\vspace{0.1em}
\begin{remark}
Here we consider the adapted version of the smallness condition \eqref{smallness} and of the entire Theorem \ref{theoMU} according to the new regime considered, where $\frac{1}{\rm Bo}=\frac{\sqrt{\mu}}{\rm bo}.$
\end{remark}

\begin{proof}
	Analogously to the proof of Theorem \ref{conv_theo}, we implement an energy estimate for \eqref{eq_diff_ref} and we get for $f=\zeta^{\rm MU}_\mu -\zeta^{\rm app}_{\rm ref}$ the following inequality:
	\begin{equation*}\begin{aligned}
	 \frac{d}{dt} | f|_{0,\nu t} +\tfrac{1}{32}(\tfrac{\sqrt{\mu}}{\rm bo}+\tfrac{\mu}{3}-1)|f|_{4,\nu t}\leq & \ \tfrac{\eps\sqrt{\mu}}{\rm bo}|N_1(f,\zeta^{\rm app}_{\rm ref})|_{0,\nu t} +\eps |N_2(f,\zeta^{\rm app}_{\rm ref})|_{0,\nu t} \\[5pt]+&\eps\mu |N_3(f,\zeta^{\rm app}_{\rm ref})|_{0,\nu t} +  \tfrac{1}{\eps}|N(\zeta^{\rm MU}_\mu, \widetilde{\phi}^\mu_{\rm ref})|_{0,\nu t}.
	\end{aligned}
	\end{equation*}
	From \eqref{N1N2_est}, \eqref{N3_est} and \eqref{Nphimu_ref_est}, it yields
	\begin{equation}\begin{aligned}
	\frac{d}{dt} | f|_{0,\nu t} +&\tfrac{1}{32}(\tfrac{\sqrt{\mu}}{\rm bo}+\tfrac{\mu}{3}-1)|f|_{4,\nu t}\\[5pt]&\leq  R_1(\zeta^{\rm app}_{\rm ref},\zeta^{\rm MU}_\mu)|f|_{4,\nu t} + R_2(\zeta^{\rm app}_{\rm ref},\zeta^{\rm MU}_\mu)|f|_{0,\nu t},\end{aligned}
	\end{equation} with
	\begin{align*}
		R_1(\zeta^{\rm app}_{\rm ref},\zeta^{\rm MU}_\mu)= 	2\eps\left(\tfrac{\sqrt{\mu}}{\rm bo}+ 1 \right)|\zeta^{\rm app}_{\rm ref}|_{0,\nu t } + \mu F_1(\zeta^{\rm app}_{\rm ref},\zeta^{\rm MU}_\mu) 
	\end{align*}
	and 
	\begin{align*}
			R_2(\zeta^{\rm app}_{\rm ref},\zeta^{\rm MU}_\mu)= & \left(4\eps\left(\tfrac{\sqrt{\mu}}{\rm bo}+ 1 \right) + \mu F_2(\zeta^{\rm app}_{\rm ref},\zeta^{\rm MU}_\mu)\right)|\zeta^{\rm app}|_{4,\nu t }\\[5pt]&+ 2\eps\left(\tfrac{\sqrt{\mu}}{\rm bo}+ 1 \right)|\zeta^{\rm MU}_\mu|_{4,\nu t }.
	\end{align*}
One can observe that the required smallness assumption \eqref{smallness} for the initial data implies the smallness assumption \eqref{smallnessTF_ref}. Therefore we can use 
both estimates \eqref{energy_estimate_ref} for $\zeta^{\rm app}_{\rm ref}$ and  \eqref{HS_energy_ineq} for $\zeta^{\rm MU}_\mu$, the fact that $\mu<1$ and $\rm bo <1$ in order to obtain that
\begin{align*}
R_1(\zeta^{\rm app}_{\rm ref},\zeta^{\rm MU}_\mu)\leq \tfrac{1}{64}(\tfrac{\sqrt{\mu}}{\rm bo}+\tfrac{\mu}{3}-1),
\end{align*} which gives 
	\begin{equation*}\begin{aligned}
	\frac{d}{dt} | f|_{0,\nu t} +\tfrac{1}{64}(\tfrac{\sqrt{\mu}}{\rm bo}+\tfrac{\mu}{3}-1)|f|_{4,\nu t}\leq R_2(\zeta^{\rm app}_{\rm ref},\zeta^{\rm MU}_\mu)|f|_{0,\nu t}+ \mu^{3/2} R^\mu_{\rm ref}
	\end{aligned}
	\end{equation*}
 From Theorem \ref{theoTF} and Theorem \ref{theoMU} we know that both  $\zeta^{\rm app}_{\rm ref}$ and $(\zeta^{\rm MU}_\mu)_{0<\mu<1}$ are uniformly bounded in $L^\infty([0,T];\mathbb{A}^0_{\nu t})\cap L^1([0,T];\mathbb{A}^4_{\nu t})$, which implies that $R_2(\zeta^{\rm app}_{\rm ref},\zeta^{\rm MU}_\mu)$ is uniformly bounded in $L^1([0,T];\mathbb{A}^4_{\nu t})$. Moreover, from Proposition \ref{propNphimu_ref}  we know that if $\zeta_{\mathrm{in},\mu}\in  \mathbb{A}^3_0$ then  $R^\mu_{\rm ref}$ is uniformly bounded in $L^1([0,T])$. We look for a convergence result for small values of the parameter $\mu$ but at the same time we need to guarantee that the condition $\tfrac{\sqrt{\mu}}{\rm bo}+\tfrac{\mu}{3}-1>0$ holds true. This implies that the values of $\mu$ cannot be taken arbitrarily small. Indeed, denoting $y=\sqrt{\mu}>0$, the previous condition reads
	$$y^2+\tfrac{3}{\rm bo}y -3>0,$$
	whose solution is the interval $(y^\ast,+\infty)$ for some $y^\ast\in (0,1)$ if $\mathrm{bo}$ is small enough. Let us hence define $\mu_0:=(y^\ast +\eps)^2$ for $\varepsilon>0$ arbitrarily small. Therefore, using Gronwall's inequality we obtain  for $\mu\in[\mu_0,1)$
	\begin{equation*}\begin{aligned}
		|f(t)|_{0,\nu t}  \ +  \  &\tfrac{1}{64}\left(\tfrac{\sqrt{\mu}}{\rm bo}+\tfrac{\mu}{3}-1\right)\int_{0}^{t} |f(t')|_{4,\nu t'} dt' \\[5pt]&\leq\mu^{3/2}\left( \int_{0}^{t} R^\mu_{\rm ref}(t')dt' \right)e^{ \int_{0}^{t}R_2(\zeta^{\rm app}_{\rm ref},\zeta^{\rm MU}_\mu) (t')dt'}\leq C \mu^{3/2} 
	\end{aligned}
	\end{equation*}
	for all $t\in[0,T]$ and some constant $C>0$ independent of $\mu$. Notice that in the previous inequality we have used that $f(0)=0$. The error estimate \eqref{stability_ref} easily follows after taking the supremum over $[0,T]$.
\end{proof}
\begin{remark}
	Since $\|f\|_{L^\infty(\mathbb{T})}\leq |f|_{0,\nu t }$, it follows from \eqref{stability} that for all $T>0$ and $\mu\in [\mu_0,1)$ we have  $$	\|\zeta^{\rm MU}_\mu -\zeta^{\rm app}_{\rm ref}\|_{L^\infty([0,T]\times \mathbb{T})}\leq   C\mu^{3/2}.$$
\end{remark}
 \vspace{0.5em}
 	\subsubsection{\textbf{The stable case}}
 	 
 	As we have previously seen, the rigorous justification of the thin film approximation at order $O(\mu^{3/2})$ in the unstable case is possible for small values of $\mu$ that are bounded from below and cannot be taken arbitrarily small as one would expect. This impossibility is due to the unstable configuration we are considering. More precisely, it is due to the sign of the gravity term in \eqref{O(mu)} that is responsible for the appearance of $-1$ in the coefficient of the fourth-order norm in the left hand side of \eqref{apriori_ref}. Then, we ask for the coefficient to be strictly positive in order to use the parabolic structure of the equation. This condition forces to impose a lower bound on the parameter $\mu$.  We now consider the refined thin film approximation in the stable configuration where the sign of the gravity terms in \eqref{O(mu)} changes, namely 
 	 		\begin{equation}\label{O(mu)stable}
 	 	\begin{aligned}
 	 	\partial_t  \zeta + \partial_x\left((1+\eps \zeta)(-\partial_x \zeta + \tfrac{\sqrt{\mu}}{\mathrm{bo}}\partial_{xxx} \zeta)\right)  - \tfrac{\mu}{3}\partial_{xx}\left((1+\eps\zeta)^3 \partial_{xx}\zeta\right)=0. 
 	 	\end{aligned}
 	 	\end{equation}
 	 Analogously to the proof of Theorem \ref{theoO(mu)}, we can write \eqref{O(mu)stable} in Fourier as 
 	 \begin{equation*}
 	 \partial_t \widehat{\zeta} + \left(\tfrac{\sqrt{\mu}}{\mathrm{bo}} -\tfrac{\mu}{3}\right)|n|^4 \widehat{\zeta} + |n|^2 \widehat{\zeta} = \widehat{\rm NL_2(\zeta)} + \widehat{\rm NL_3(\zeta)} +\widehat{\rm NL_4(\zeta)}
 	 \end{equation*} for slightly different nonlinear terms $\mathrm{NL}_2(\zeta)$, $\mathrm{NL}_3(\zeta)$, $\mathrm{NL}_4(\zeta)$ and it yields
 	 \begin{equation*}
 		\partial_t |\widehat{\zeta}| + \left(\tfrac{\sqrt{\mu}}{\mathrm{bo}} -\tfrac{\mu}{3}\right)|n|^4 |\widehat{\zeta}|  +|n|^2 |\widehat{\zeta}| \leq |\widehat{\rm NL_2(\zeta)} |+ |\widehat{\rm NL_3(\zeta)}| +|\widehat{\rm NL_4(\zeta)}|.\\[5pt]
 	 \end{equation*} Taking advantage of the positive sign of the gravity term 
$|n|^2 |\widehat{\zeta}|$ and
 choosing $$\nu\in \left[0,\tfrac{1}{2}\left(\tfrac{\sqrt{\mu}}{\mathrm{bo}} -\tfrac{\mu}{3}\right)\right)$$ we get
 \begin{equation*}
 \frac{d}{dt} | \zeta|_{0,\nu t} +\tfrac{1}{32}\left(\tfrac{\sqrt{\mu}}{\mathrm{bo}} -\tfrac{\mu}{3}\right)|\zeta|_{4,\nu t}\leq |\rm NL_2(\zeta)|_{0,\nu t}+ |\rm NL_3(\zeta)|_{0,\nu t} + |\rm NL_4(\zeta)|_{0,\nu t}.
 \end{equation*}Then, by using the same argument of the proof of Theorem \ref{conv_ref_theo}, we get the following existence result:

 \begin{theorem}\label{theoO(mu)stable}
 	Let $\mathrm{bo}>0$, $0<\mu<1$ such that $\tfrac{\sqrt{\mu}}{\mathrm{bo}} -\tfrac{\mu}{3}>0$ and $\zeta_0\in \wiener[0]_0(\mathbb{T})$ be a zero-mean function such that 
 	\begin{equation}\label{smallnessTF_ref_stable}
 	|\zeta_0|_{0,0}< \tfrac{\tfrac{\sqrt{\mu}}{\mathrm{bo}} -\tfrac{\mu}{3}}{128\eps \left(\tfrac{\sqrt{\mu}}{\mathrm{bo}}+\tfrac{55}{6}\mu+1\right)}.\end{equation} Then there exists a global weak solution to \eqref{O(mu)stable} in the sense of Definition \ref{weakdefO(mu)} with initial data $\zeta_0$ which becomes instantaneously analytic in a growing strip in the complex plane. In particular, for $$\nu\in \left[0,\tfrac{1}{2}\left(\tfrac{\sqrt{\mu}}{\mathrm{bo}} -\tfrac{\mu}{3}\right)\right),$$ we have  $$\zeta\in L^\infty([0,T];\wiener[0]_{\nu t})\cap L^1([0,T];\wiener[4]_{\nu t})$$
 	for all $T>0$ and the solution $\zeta$ satisfies for any $t\in[0,T]$ the energy inequality
 	\begin{equation}\label{energy_estimate_ref_stable}
 	|\zeta(t)|_{0,\nu t} + \tfrac{1}{64} \left(\tfrac{\sqrt{\mu}}{\mathrm{bo}} -\tfrac{\mu}{3}\right)\int_0^t|\zeta(t')|_{4,\nu t'}dt'\leq	|\zeta_0|_{0,0}
 	\end{equation} together with the decay
 	\begin{equation}\label{decay_ref_stable}
 	|\zeta(t)|_{0,\nu t} \leq 	|\zeta_0|_{0,0} \ e^{-\frac{t}{64}\left(\tfrac{\sqrt{\mu}}{\mathrm{bo}} -\tfrac{\mu}{3}\right)}.
 	\end{equation}
 \end{theorem}\vspace{0.5em}
Finally, we obtain a rigorous justification of the refined thin film approximation in the stable case without any lower bound for $\mu$:
\begin{theorem}\label{conv_ref_theo_stable}
Let $\mathrm{bo>0}$ and $\zeta_{\mathrm{in},\mu}$ be a zero-mean function in $\mathbb{A}^3_0$ satisfying the smallness assumption \eqref{smallness}. Let $\zeta^{\rm MU}_\mu$  and $\zeta^{\rm app}_{\rm ref}$ be the global weak solutions respectively to \eqref{heleshaw_dimless} and to \eqref{O(mu)stable} both with initial data $\zeta_{\mathrm{in},\mu}.$ Then, for any $T>0$ there exist $C>0$ independent of $\mu$ and $\mu_1\in(0,1)$ such that for $\mu\in (0,\mu_1)$
\begin{equation}\label{stability_ref_stable}
\|\zeta^{\rm MU}_\mu -\zeta^{\rm app}_{\rm ref}\|_{L^\infty([0,T]; \ \mathbb{A}^0_{\nu t})}+ \tfrac{1}{64}\left(\tfrac{\sqrt{\mu}}{\mathrm{bo}} -\tfrac{\mu}{3}\right)	\|\zeta^{\rm MU}_\mu -\zeta^{\rm app}_{\rm ref}\|_{L^1([0,T]; \ \mathbb{A}^4_{\nu t})}\leq  C\mu^{3/2}.
\end{equation}
\end{theorem}
 \vspace{1em}

\subsection*{Acknowledgments.} 		
E. Bocchi and F. Gancedo were partially supported by the ERC through the Starting Grant project H2020-EU.1.1.-639227, by the grant P20-00566 of Junta de Andalucía and by the grant EUR2020-112271 (Spain). F. Gancedo was partially supported by MINECO grant RED2018-102650-T (Spain). The authors would like to thank Peter Constantin and Rafael Granero-Belinchón for helpful discussions.

\bibliographystyle{siam}
\bibliography{references}
	
	\vspace{.1in}

\begin{tabular}{l}
	\textbf{Edoardo Bocchi} \\
	{\small Departamento de An\'alisis Matem\'atico \& IMUS}\\
	{\small Universidad de Sevilla} \\ 
	{\small C/ Tarfia, s/n}\\ 
	{\small Campus Reina Mercedes, 41012, Sevilla, Spain} \\ 
	{\small Email: ebocchi@us.es}
\end{tabular}

	\vspace{.1in}

\begin{tabular}{l}
	\textbf{Francisco Gancedo} \\
	{\small Departamento de An\'alisis Matem\'atico \& IMUS}\\
	{\small Universidad de Sevilla} \\ 
	{\small C/ Tarfia, s/n}\\ 
	{\small Campus Reina Mercedes, 41012, Sevilla, Spain} \\ 
	{\small Email: fgancedo@us.es}
\end{tabular}

\end{document}